\documentclass[12pt]{scrartcl}

\usepackage[]{amsmath, amssymb,amsfonts,amsthm,mathtools,braket,color,enumerate,stmaryrd}

\usepackage{hhline}
\usepackage{url}
\usepackage{here}
\usepackage{tikz}
\usepackage{tikz-cd}
\usepackage{hyperref}
\hypersetup{
  colorlinks   = true, 
  urlcolor     = blue, 
  linkcolor    = blue, 
  citecolor   = red 
}
\usetikzlibrary{decorations}
\usetikzlibrary{arrows}
\tikzstyle{v} = [circle, draw, inner sep=2pt, minimum size=3pt, fill=black]
\tikzstyle{l} = [rectangle, draw, rounded corners]

\theoremstyle{plain}
\newtheorem{theorem}{Theorem}[section]
\newtheorem{lemma}[theorem]{Lemma}
\newtheorem{proposition}[theorem]{Proposition}
\newtheorem{corollary}[theorem]{Corollary}

\theoremstyle{definition}
\newtheorem{definition}[theorem]{Definition}

\newtheorem{example}[theorem]{Example}

\newtheorem{question}[theorem]{Question}
\newtheorem{remark}[theorem]{Remark}

\DeclareMathOperator{\Ima}{Im}
\DeclareMathOperator{\coker}{coker}
\DeclareMathOperator{\rank}{rank}
\DeclareMathOperator{\lcm}{lcm}
\DeclareMathOperator{\id}{id}
\DeclareMathOperator{\Spec}{Spec}
\DeclareMathOperator{\Frac}{Frac}
\DeclareMathOperator{\Ann}{Ann}
\DeclareMathOperator{\ord}{ord}

\DeclareMathOperator{\codim}{codim}
\newcommand{\A}{\mathcal{A}}

\begin{document}

\title{The Characteristic Quasi-Polynomials of Hyperplane Arrangements over Residually Finite Dedekind Domains}
\author{
Masamichi Kuroda
\thanks{
Faculty of Engineering, Nippon Bunri University, Oita 870-0316, Japan. 
E-mail:kurodamm@nbu.ac.jp
} \and
Shuhei Tsujie
\thanks{Department of Mathematics, Hokkaido University of Education, Asahikawa, Hokkaido 070-8621, Japan. 
Email:tsujie.shuhei@a.hokkyodai.ac.jp}
}

\date{}

\maketitle

\begin{abstract}
Kamiya, Takemura, and Terao initiated the theory of the characteristic quasi-polynomial of an integral arrangement, which is a function counting the elements in the complement of the arrangement modulo positive integers. 

They gave a period of the characteristic quasi-polynomial, called the LCM-period, and  showed that the first constituent of the characteristic quasi-polynomial coincides with the characteristic polynomial of the corresponding hyperplane arrangement. 

Recently, Liu, Tran, and Yoshinaga showed that the last constituent of the characteristic quasi-polynomial coincides with the characteristic polynomial of the corresponding toric arrangement. 

In addition, by using the theory of toric arrangements, Higashitani, Tran, and Yoshinaga proved that the LCM-period is the minimum period of the characteristic quasi-polynomial. 

In this paper, we study an arrangements over a Dedekind domain such that every residue ring with a nonzero ideal is finite and give algebraic generalizations of the above results.
\end{abstract}



\tableofcontents

\section{Introduction}


For a positive integer $ \ell $, let $\A = \{ c_1, \dots , c_n \} \subseteq \mathbb{Z}^{\ell}$ be a finite subset consisting of nonzero integral column vectors. 
Roughly speaking, an \textbf{arrangement} means a finite collection of subspaces of codimension $ 1 $ of a space. 
We can define three kinds of arrangements from $ \mathcal{A} $ as follows: 

The first is the \textbf{hyperplane arrangement} $\A (\mathbb{R}) = \{ H_{1}, \dots, H_{n} \}$ in the vector space $\mathbb{R}^{\ell}$ consisting of hyperplanes 
\begin{align*}
H_j \coloneqq \Set{ \boldsymbol{x} = (x_1, \dots, x_\ell) \in \mathbb{R}^{\ell} | \boldsymbol{x} c_j = 0 } \quad (j \in [n] \coloneqq \{ 1, \dots, n \}). 
\end{align*} 

The second is the \textbf{$q$-reduced arrangement} $\A (\mathbb{Z} / q \mathbb{Z}) = \{ H_{1, q}, \dots, H_{n, q} \}$ in $(\mathbb{Z} / q \mathbb{Z})^{\ell}$ for any positive integer $q$, where 
\begin{align*}
H_{j, q} \coloneqq \Set{ [\boldsymbol{x}]_{q} \in (\mathbb{Z} / q \mathbb{Z})^{\ell} | \boldsymbol{x} c_j \equiv 0 \pmod{q}  } \quad
(j \in [n]) 
\end{align*} 
and $ [\boldsymbol{x}]_{q} $ denotes the equivalence class of $ \boldsymbol{x} $. 

The third is the \textbf{toric arrangement} $\A (\mathbb{C}^{\times}) = \{ T_1, \dots , T_n \}$ in the algebraic torus $(\mathbb{C}^{\times})^{\ell}$, 
where 
\begin{align*}
T_j \coloneqq \Set{ (t_1, \dots, t_n) \in (\mathbb{C}^{\times})^{\ell} | \prod_{i=1}^{\ell} t_i^{c_{ij}} = 1 } \quad (j \in [n])
\end{align*} 
and $ c_{ij} $ denotes the $ i $-th coefficient of the column vector $ c_{j} $.

Combinatorics of $ \mathcal{A}(\mathbb{R}) $ and $ \mathcal{A}(\mathbb{C}^{\times}) $ is described as the posets of the intersections defined as follows: 
\begin{align*}
L (\A (\mathbb{R})) &\coloneqq \Set{H_J | J \subseteq [n]},    \\
L (\A (\mathbb{C}^{\times})) &\coloneqq 
\Set{Z | J \subseteq [n], \ \mbox{$Z$ is a connected component of $T_J$} }, 
\end{align*}
where $ H_J \coloneqq \bigcap_{j \in J} H_j $ and $ T_J \coloneqq \bigcap_{j \in J} T_j $, and the partial order defined by reverse inclusion for each poset. 
Note that $ H_{\varnothing} = \mathbb{R}^{\ell} $ and $ T_{\varnothing} = (\mathbb{C}^{\times})^{\ell} $ are the minimal elements. 

Main concerns for study of arrangements are relations among combinatrics, algebra, and geometry of arrangements. 
The characteristic polynomials of the posets play important rolls, which are defined by using the M\"{o}bius functions on posets. 
Let $ P $ be a finite poset with unique minimal element $ \hat{0} $. 
The \textbf{M\"{o}bius function} $ \mu $ on $ P $ is defined recursively by 
\begin{align*}
\mu(\hat{0}) \coloneqq 1
\qquad \text{ and } \qquad 
\mu(Z) \coloneqq -\sum_{Y < Z} \mu(Y) \text{ for } Z \neq \hat{0}. 
\end{align*}

The \textbf{characteristic polynomials} $ \chi_{\A (\mathbb{R})} (t) $ and $ \chi_{\A (\mathbb{C}^{\times})} $ are defined by  
\begin{align*}
\chi_{\A (\mathbb{R})} (t) &\coloneqq \sum_{Z \in L (\A (\mathbb{R}))} \mu (Z) t^{\dim Z},  \\
\chi_{\A (\mathbb{C}^{\times})} (t) &\coloneqq \sum_{Z \in L (\A (\mathbb{C}^{\times}))} \mu (Z) t^{\dim Z}. 
\end{align*}

The \textbf{complement} of an arrangement is the complements of the union of the members of the arrangement in the ambient space. 
Each connected component of the complement of $ \mathcal{A}(\mathbb{R}) $ is called a chamber. 
Zaslavsky \cite{zaslavsky1975facing} proved that the numbers of chambers and bounded chambers coincide with $ |\chi_{\mathcal{A}(\mathbb{R})}(-1)| $ and $ |\chi_{\mathcal{A}(\mathbb{R})}(1)| $. 
Orlic and Solomon \cite{orlik1980combinatorics-im} proved that $ \chi_{\A (\mathbb{R})} (t) $ is equivalent to the Poincar\'e polynomial of the complement of the complexification of $ \mathcal{A}(\mathbb{R}) $. 
Moci \cite[Corollary 5.12]{moci2012tutte-totams} showed that $ \chi_{\A (\mathbb{C}^{\times})}(t) $ is equivalent to the Poincar\'e polynomial of the complement of $ \mathcal{A}(\mathbb{C}^{\times}) $. 
These facts present a strong association between combinatorics and geometry of arrangements. 

Athanasiadis \cite[Theorem 2.2]{athanasiadis1996characteristic-aim} provided a method to compute the characteristic polynomial of an integral arrangement by counting the points of the complement of $ \mathcal{A}(\mathbb{Z}/p\mathbb{Z}) $ for large enough prime numbers $ p $.  
Athanasiadis \cite[Theorem 2.1]{athanasiadis1999extended-joac} also proved that the characteristic polynomial can be computed by counting the points of the complement of $ \mathcal{A}(\mathbb{Z}/q\mathbb{Z}) $ for large enough integers $ q $ relatively prime a constant which depends only on $ \mathcal{A} $. 

Kamiya, Takemura, and Terao developed Athanasiadis' method by considering the complement of $ \mathcal{A}(\mathbb{Z}/q\mathbb{Z}) $ for all positive integers $ q $ as follows. 
For a nonempty subset $ J = \{j_{1}, \dots, j_{k}\} \subseteq [n] $, suppose that the matrix $ C_{J} \coloneqq (c_{j_{1}} \ \cdots \ c_{j_{k}}) $ has the Smith normal form 
\begin{align*}
\begin{pmatrix}
d_{J,1} & 0 & \cdots & 0 & \cdots & \cdots & 0 \\
0 & d_{J,2} & & \vdots & & & \vdots \\
\vdots & & \ddots & 0 & & & \\
0 & \cdots & 0 & d_{J, r(J)} & & & \\
\vdots & &&& 0 & & \\
\vdots &&&&& \ddots & \vdots \\
0 & \cdots & && &\cdots & 0
\end{pmatrix}, 
\end{align*}
where $ d_{J,i} $ is a positive integer such that $ d_{J,i} $ divides $ d_{J,i+1} $. 
Define $ \rho_{\mathcal{A}} \in \mathbb{Z}_{>0} $ by 
\begin{align*}
\rho_{\mathcal{A}} \coloneqq \lcm\Set{d_{J, r(J)} | \varnothing \neq J \subseteq [n]}. 
\end{align*}


\begin{theorem}[{Kamiya-Takemura-Terao \cite[Therem 2.4]{kamiya2008periodicity-joac}}] \label{KTT081} 
Let $ M (\A(\mathbb{Z} / q \mathbb{Z}) ) \coloneqq (\mathbb{Z}/q\mathbb{Z})^{\ell}\setminus\bigcup_{J \subseteq [n]}H_{J,q} $ denote the complement of $ \mathcal{A}(\mathbb{Z}/q\mathbb{Z}) $. 
Then the function $ \left|M (\A(\mathbb{Z} / q \mathbb{Z}) )\right| $ is a monic integral quasi-polynomial in $ q \in \mathbb{Z}_{>0} $ with a period $ \rho_{\mathcal{A}} $. 
Namely, there exist monic polynomials $ f_{\mathcal{A}}^{k}(t) \in \mathbb{Z}[t] \ (1 \leq k \leq \rho_{\A}) $ such that $ f_{\mathcal{A}}^{k}(q) = \left|M (\A(\mathbb{Z} / q \mathbb{Z}) )\right| $ if $ q \equiv k \pmod{\rho_{\A}} $. 
Furthermore, the quasi-polynomial has the \textbf{GCD-property}, that is, $ f_{\mathcal{A}}^{k}(t) = f_{\mathcal{A}}^{k^{\prime}}(t) $ when $ \gcd(k, \rho_{\mathcal{A}}) = \gcd(k^{\prime}, \rho_{\mathcal{A}}) $. 
\end{theorem}

\begin{definition}
We call the quasi-polynomial $ \chi^{\mathrm{quasi}}_{\A} (q) \coloneqq \left|M (\A(\mathbb{Z} / q \mathbb{Z}) )\right| $ the \textbf{characteristic quasi-polynomial} of $ \mathcal{A} $. 
The period $ \rho_{\mathcal{A}} $ is called the \textbf{LCM-period}. 
The polynomial $ f_{\mathcal{A}}^{k}(t) $ is said to be the \textbf{$ k $-constituent} of $ \chi^{\mathrm{quasi}}_{\A} (q) $. 
\end{definition}

Interestingly enough, each constituent of the characteristic quasi-polynomial has a combinatorial interpretation. 

\begin{theorem}[{Kamiya-Takemura-Terao \cite[Therem 2.5]{kamiya2008periodicity-joac}}] \label{KTT082}
The $ 1 $-constituent of the characteristic quasi-polynomial of $ \mathcal{A} $ is the characteristic polynomial of the hyperplane arrangement $ \mathcal{A}(\mathbb{R}) $. 
Namely, $f^{1}_{\A} (t) = \chi_{\A (\mathbb{R})} (t)$. 
\end{theorem}

\begin{theorem}[{Liu--Tran--Yoshinaga \cite[Corollary 5.6]{liu2021$g$-tutte-imrn}}] \label{LTY17}
The $ \rho_{\A} $-constituent of the characteristic quasi-polynomial of $ \mathcal{A} $ is the characteristic polynomial of the toric arrangement $ \mathcal{A}(\mathbb{C}^{\times}) $. 
Namely, $f^{\rho_{\mathcal{A}}}_{\A} (t) = \chi_{\A (\mathbb{C}^{\times})} (t)$. 
\end{theorem}

Tran and Yoshinaga \cite{tran2019combinatorics-joctsa} gave combinatorial interpretation for the other constituents and unified Theorem \ref{KTT082} and Theorem \ref{LTY17}. 
For a positive integer $k$, define 
\begin{align*}
L (\A (\mathbb{C}^{\times}))[k] \coloneqq \Set{Z \in L (\A (\mathbb{C}^{\times})) | \mbox{$Z$ contains a $k$-torsion element}}, 
\end{align*} 
\begin{align*}
\chi_{\A (\mathbb{C}^{\times})}^{k} (t) \coloneqq \sum_{Z \in L (\A (\mathbb{C}^{\times}))[k]} \mu (Z) t^{\dim Z}. 
\end{align*}
Note that $\chi_{\A (\mathbb{C}^{\times})}^{1} (t) = \chi_{\A (\mathbb{R})} (t) $ and $ \chi_{\A (\mathbb{C}^{\times})}^{\rho_{\mathcal{A}}} (t) = \chi_{\A (\mathbb{C}^{\times})} (t)$. 

\begin{theorem}[{Tran-Yoshinaga \cite[Corollary 4.8]{tran2019combinatorics-joctsa}}] \label{TY19}
Let $ 1 \leq k \leq \rho_{\mathcal{A}}$. 
Then $f^{k}_{\A} (t) = \chi_{\A (\mathbb{C}^{\times})}^{k} (t)$. 
\end{theorem}

For a decade, it was an open problem whether the LCM-period is minimum or not. 
Recently, using Theorem \ref{TY19}, Higashitani, Tran, and Yoshinaga gave an affirmative answer for central arrangements. 

\begin{theorem}[{Higashitani-Tran-Yoshinaga \cite[Theorem 1.2]{higashitani2022period-imrn}}] \label{HTY21}
The LCM-period $\rho_{\mathcal{A}}$ is the minimum period of the characteristic quasi-polynomial $\chi^{\mathrm{quasi}}_{\A} (q)$. 
\end{theorem}

\begin{remark}
The characteristic quasi-polynomial its LCM-period can be considered for non-central arrangements \cite{kamiya2011periodicity-aoc}. 
Higashitani, Tran, and Yoshinaga \cite{higashitani2022period-imrn} also studied non-central arrangements such that the LCM-periods are not minimum. 
\end{remark}

The properties of the characteristic quasi-polynomial highly depend on ring-theoretic properties of $ \mathbb{Z} $. 
In particular, the theory of Smith normal forms and their invariant factors are essential. 
It is a natural question to ask what occurs if we replace the ring $ \mathbb{Z} $ with another ring $ \mathcal{O} $. 
Let $ \mathcal{A} $ be a finite subset of $ \mathcal{O}^{\ell} $ and we will investigate the function $ \left| M(\mathcal{A}(\mathcal{O}/\mathfrak{a})) \right| $ in nonzero ideals of $ \mathfrak{a} $. 
In order to do this, $ \left| M(\mathcal{A}(\mathcal{O}/\mathfrak{a})) \right| $ must be finite for every $ \mathfrak{a} $. 
Moreover, the theory of finitely generated modules of Dedekind domains is considered to be a direct generalization of the theory of Smith normal forms. 

For the reasons above, in this article, we will consider the function $ \left| M(\mathcal{A}(\mathcal{O}/\mathfrak{a})) \right| $ under the following convention. 
\begin{itemize}
\item $ \mathcal{O} $ is a Dedekind domain, that is, an integral domain in which every nonzero proper ideal has a unique factorization into prime ideals. 
\item $ \mathcal{O}/\mathfrak{a} $ is a finite ring for every nonzero ideal $ \mathfrak{a} $. 
\end{itemize}
Such a ring $ \mathcal{O} $ is called a \textbf{residually finite Dedekind domain} or a \textbf{Dedekind domain with the finite norm property}. 
The ring $ \mathbb{Z} $ is a typical example of a residually finite Dedekind domain. 
More generally, it is well known that the ring of integers of an algebraic field is a residually finite Dedekind domain. 
The polynomial ring $ \mathbb{F}_{q}[t] $ over a finite field $ \mathbb{F}_{q} $ and the ring of $ p $-adic integers $ \mathbb{Z}_{p} $ are also examples of residually finite Dedekind domains. 

The purpose of this paper is to give algebraic generalizations of Theorem \ref{KTT081}, \ref{KTT082}, \ref{LTY17}, \ref{TY19}, and \ref{HTY21} for a finite subset $ \mathcal{A} = \{c_{1}, \dots, c_{n}\} $ of $ \mathcal{O}^{\ell} $. 
Let $ K \coloneqq \Frac \mathcal{O} $ denote the fraction filed of $ \mathcal{O} $. 
We will consider the following three kinds of arrangements: 

The first is the \textbf{hyperplane arrangement} $ \mathcal{A}(K) = \{H_{1}, \dots, H_{n}\} $, where 
\begin{align*}
H_{j} \coloneqq \Set{\boldsymbol{x} \in K^{\ell} | \boldsymbol{x}c_{j} = 0}. 
\end{align*}

The second is the \textbf{$ \mathfrak{a} $-reduced arrangement} $ \mathcal{A}(\mathcal{O}/\mathfrak{a}) = \{H_{1,\mathfrak{a}}, \dots, H_{n, \mathfrak{a}}\} $ for every nonzero ideal $ \mathfrak{a} \subseteq \mathcal{O} $, where 
\begin{align*}
H_{j, \mathfrak{a}} \coloneqq \Set{[\boldsymbol{x}]_{\mathfrak{a}} \in (\mathcal{O}/\mathfrak{a})^{\ell} | \boldsymbol{x}c_{j} \equiv 0 \pmod{\mathfrak{a}}}
\end{align*}
and $ [\boldsymbol{x}]_{\mathfrak{a}} $ denotes the equivalent class of $ \boldsymbol{x} $. 

The third is the \textbf{torsion arrangement} $ \mathcal{A}(K/\mathcal{O}) = \{T_{1}, \dots, T_{n}\} $, where 
\begin{align*}
T_{j} \coloneqq \Set{\pi(\boldsymbol{x}) \in (K/\mathcal{O})^{\ell} | \boldsymbol{x}c_{j} \equiv 0 \pmod{\mathcal{O}}}
\end{align*}
and $ \pi \colon K \to K/\mathcal{O} $ denotes the canonical projection. 
Note that $ \mathbb{Q} = \Frac \mathbb{Z} $ and if $ \mathcal{A} \subseteq \mathbb{Z}^{\ell} $, then $ \mathcal{A}(\mathbb{R}) $ and $ \mathcal{A}(\mathbb{C}^{\times}) $ are combinatorially equivalent to $ \mathcal{A}(\mathbb{Q}) $ and $ \mathcal{A}(\mathbb{Q}/\mathbb{Z}) $, respectively. 
The torsion arrangement $ \mathcal{A}(K/\mathcal{O}) $ will work as a substitute of the toric arrangement. 

Let $ I(\mathcal{O}) $ be the set consisting nonzero ideals of $ \mathcal{O} $. 
For every $ \mathfrak{a} \in I(\mathcal{O}) $, define the \textbf{complement} $ M(\mathcal{A}(\mathcal{O}/\mathfrak{a})) $ by 
\begin{align*}
M(\mathcal{A}(\mathcal{O}/\mathfrak{a})) \coloneqq \left(\mathcal{O}/\mathfrak{a}\right)^{\ell} \setminus \bigcup_{j = 1}^{n} H_{j, \mathfrak{a}}. 
\end{align*}

\begin{definition}
The function $ \chi_{\mathcal{A}}^{\mathrm{quasi}} \colon I(\mathcal{O}) \to \mathbb{Z} $ determined by $ \chi_{\mathcal{A}}^{\mathrm{quasi}}(\mathfrak{a}) \coloneqq |M(\mathcal{A}(\mathcal{O}/\mathfrak{a})) | $ is called the \textbf{characteristic quasi-polynomial} of $ \mathcal{A} $. 
\end{definition}

The organization of this article is as follows: 
In Section \ref{Quasi-polynomials with GCD-property}, we introduce and study an algebraic generalization of a quasi-polynomial with GCD-property (Definition \ref{quasi-polynomial with GCD-property}), which is a function on nonzero proper ideals of $ \mathcal{O} $ represented by using ``periodically" finitely many polynomials, called constituents. 
A ``period" of this generalization is a nonzero ideal and every constituent corresponds with a factor of a period.

In Section \ref{Characteristic quasi-polynomial}, 
we will show that the characteristic quasi-polynomial $ \chi_{\mathcal{A}}^{\mathrm{quasi}} $ is a quasi-polynomial in the sense of the previous section with the LCM-period $ \rho_{\mathcal{A}} $ (Theorem \ref{main thm}), which generalize Theorem \ref{KTT081}. 
Moreover, we will give a generalization of Theorem \ref{KTT082} which states that $ \langle 1 \rangle $-constituent of $ \chi^{\mathrm{quasi}}_{\A} (\mathfrak{a}) $ coincides with the characteristic polynomial $ \chi_{\mathcal{A}(K)}(t) $ of the hyperplane arrangement $ \mathcal{A}(K) $ (Theorem \ref{main thm1.5}), where $ \langle 1 \rangle = \mathcal{O} $ stands for the unit ideal.


In Section \ref{torsion arrangement}, 
we will study the torsion arrangement $\A (K / \mathcal{O})$ and the poset $L (\A (K/\mathcal{O}))$ defined by the information of the intersections of $\A (K / \mathcal{O})$, called the poset of layers.
We will introduce the $ \kappa $-torsion subposet $L (\A (K/\mathcal{O})) [\kappa]$ and its characteristic polynomial $\chi_{\A (K/\mathcal{O})}^{\kappa} (t)$ for each factor $\kappa$ of the LCM-period $\rho_{\mathcal{A}}$ and prove that $\chi_{\A (K/\mathcal{O})}^{\kappa} (t)$ coincides with the $\kappa$-constituent of $\chi^{\mathrm{quasi}}_{\A}$ 
(Theorem \ref{main thm2}). 
This result is a generalization of 
Theorem \ref{TY19} including Theorem \ref{LTY17}.

In Section \ref{proof of minimality}, with the similar technique as the proof of Theorem \ref{HTY21}, 
we will prove that the LCM-period $\rho_{\mathcal{A}}$ is the minimum period of $\chi^{\mathrm{quasi}}_{\A}$ (Theorem \ref{main thm3}). 

In Section \ref{localizations of the base rings}, we will study behavior of the characteristic quasi-polynomials and the posets of layers under taking localization of the base rings.

In Section \ref{computing}, we will discuss general properties of the characteristic quasi-polynomials, which are generalizations of results in \cite{kamiya2011periodicity-aoc} and very useful for computing. 

Kamiya, Takemura, and Terao \cite{kamiya2011periodicity-aoc} studied the characteristic quasi-polynomial of the irreducible crystallographic root systems. 
Thanks to the results in this article, we can consider the characteristic quasi-polynomials of non-crystallographic root systems. 
In Section \ref{Examples}, 
we will give the characteristic quasi-polynomials explicitly for non-crystallographic root systems of types $\mathrm{H}_2$, $\mathrm{H}_3$, and $\mathrm{H}_4$.

\section{Quasi-polynomials on the set of nonzero ideals} \label{Quasi-polynomials with GCD-property}

Let $ \mathcal{O} $ be a residually finite Dedekind domain. 
Let $ I(\mathcal{O}) $ denote the set consisting of nonzero ideals of $ \mathcal{O} $ and $ N(\mathfrak{a}) $ the \textbf{absolute norm} of $ \mathfrak{a} \in I(\mathcal{O}) $, that is, the cardinality of the quotient ring $ \mathcal{O}/\mathfrak{a} $.  

Note that every ideal in $I(\mathcal{O}) $ factors into a product of prime ideals uniquely. 
We write $ \mathfrak{a} \mid \mathfrak{b} $ if $ \mathfrak{a} $ is a factor of $ \mathfrak{b} $, or equivalently $ \mathfrak{a} \supseteq \mathfrak{b} $. 
The sum $ \mathfrak{a} + \mathfrak{b} $ coincides with the greatest common divisor $ \gcd(\mathfrak{a}, \mathfrak{b}) $ and the intersection $ \mathfrak{a} \cap \mathfrak{b} $ coincides the least common multiple $ \lcm(\mathfrak{a}, \mathfrak{b}) $. 

\begin{definition} \label{quasi-polynomial with GCD-property}
A function $\phi \colon I(\mathcal{O})  \rightarrow \mathbb{Z}$ 
is called a \textbf{quasi-polynomial on $ I(\mathcal{O}) $}  if there exists an ideal $ \rho \in I(\mathcal{O}) $ 
and polynomials $f^{\kappa} (t) \in \mathbb{Z} [t] $ for each $\kappa \mid \rho$ such that 
for any $\mathfrak{a} \in I(\mathcal{O}) $ with $\mathfrak{a} + \rho = \kappa$,  
\begin{align*}
\phi \left( \mathfrak{a} \right) = f^{\kappa} \left( N (\mathfrak{a}) \right) .
\end{align*}
The ideal $\rho$ is called a \textbf{period} and the polynomial $ f^{\kappa} (t) $ is called the \textbf{$\kappa$-constituent} of the quasi-polynomial $\phi$ associated with $ \rho $. 
A period which divides every period is called the \textbf{minimum period} of the quasi-polynomial $\phi$. 
\end{definition}

\begin{remark}
A quasi-polynomial on $ I(\mathbb{Z}) $ is a quasi-polynomial \emph{with the GCD-property}. 
Our generalization of quasi-polynomial \emph{does not} contain quasi-polynomials without the GCD-property. 
\end{remark}

For any nonzero prime ideal $\mathfrak{p} \in I (\mathcal{O})$, we define $\ord_{\mathfrak{p}} \colon I(\mathcal{O}) \to \mathbb{Z}_{\geq 0}$ by 
$\ord_{\mathfrak{p}} (\mathfrak{a}) \coloneqq \max \Set{a \in \mathbb{Z}_{\geq 0} | \ \mathfrak{p}^a \mid \mathfrak{a}}$. 

\begin{proposition}\label{minimum period}
Let $\phi $ be a quasi-polynomial on $ I(\mathcal{O}) $. 
The following statements hold: 
\begin{enumerate}[(i)]
\item\label{minimum period 1} 
For any periods $\rho_1$ and $\rho_2$ of $\phi$, 
the sum $\rho_1 + \rho_2$ is a period of $\phi$. 
\item\label{minimum period 2} 
The minimum period of $\phi$ always exists. 
\end{enumerate}
\end{proposition}

\begin{proof}
(\ref{minimum period 1}) 
For $ i \in \{1,2\} $, let $f_i^{\kappa}(t) \in \mathbb{Z}[t]$ denote the $\kappa$-constituent of $\phi$ associated with the period $\rho_i$ for $ \kappa \mid \rho_{i} $.  
Let $\rho \coloneqq \rho_1 + \rho_2 $ and fix a divisor $ \kappa $ of $ \rho $. 
It is sufficient to show that there exists a polynomial $ f^{\kappa}(t) \in \mathbb{Z}[t] $ such that $ \mathfrak{a} + \rho = \kappa $ implies $ \phi(\mathfrak{a}) = f^{\kappa}(N(\mathfrak{a})) $. 

First suppose that $ \ord_{\mathfrak{p}}(\kappa) < \ord_{\mathfrak{p}}(\rho) $ for any nonzero prime ideal $ \mathfrak{p} $ of $ \mathcal{O} $. 
Note that this happens only when the number of prime ideals of $ \mathcal{O} $ is finite. 
If $ \mathfrak{a} + \rho = \kappa $, then $ \ord_{\mathfrak{p}}(\mathfrak{a}) = \ord_{\mathfrak{p}}(\kappa) $ for all nonzero prime ideals $ \mathfrak{p} $ in $ \mathcal{O} $ and hence $ \mathfrak{a} = \kappa $. 
Since $ \mathfrak{a}+\rho_{1} = \kappa + \rho_{1} = \kappa + \rho = \kappa $, we have $ \phi(\mathfrak{a}) = f_{1}^{\kappa}(N(\mathfrak{a})) $. 
Thus the polynomial $ f^{\kappa}(t) \coloneqq f_{1}^{\kappa}(t) $ has the desired property. 

Now, we consider the case there exists a nonzero prime ideal $ \mathfrak{p} $ of $ \mathcal{O} $ such that $\ord_{\mathfrak{p}} (\kappa) = \ord_{\mathfrak{p}} (\rho)$. 
Without loss of generality, we can assume that $\ord_{\mathfrak{p}} (\rho_1) \leq \ord_{\mathfrak{p}} (\rho_2)$. 
Suppose that $ \kappa^{\prime} $ is an ideal such that $ \kappa \mid \kappa^{\prime} $ and $ \kappa^{\prime} \mid \rho_{1} $. 
Note that $ \ord_{\mathfrak{p}}(\kappa) = \ord_{\mathfrak{p}}(\kappa^{\prime}) = \ord_{\mathfrak{p}}(\rho_{1}) $. 

Put $\nu \coloneqq \ord_{\mathfrak{p}} (\rho_2) - \ord_{\mathfrak{p}} (\rho_1) \geq 0$. 
Then for any integer $ m \geq \nu $, we have $ \mathfrak{p}^m \kappa^{\prime} + \rho_1 = \kappa^{\prime} $ and $ \mathfrak{p}^m \kappa + \rho_2 = \mathfrak{p}^{\nu}\kappa $. 
Therefore $
f_1^{\kappa^{\prime}} (N (\mathfrak{p}^m \kappa)) = \phi(\mathfrak{p}^{m} \kappa) = f_2^{\mathfrak{p}^\nu \kappa} (N(\mathfrak{p}^{m} \kappa))
$ for any $m \geq \nu$. 
Thus $ f_{1}^{\kappa^{\prime}}(t) = f_{2}^{\mathfrak{p}^{\nu}\kappa}(t) $. 
In particular, $ f_{1}^{\kappa^{\prime}}(t) = f_{1}^{\kappa}(t) $ for any ideal $ \kappa^{\prime} $ such that $ \kappa \mid \kappa^{\prime} $ and $ \kappa^{\prime} \mid \rho_{1} $.  

Suppose that $ \mathfrak{a} + \rho = \kappa $. 
Then $ \kappa^{\prime} \coloneqq \mathfrak{a} + \rho_{1} $ satisfies $ \kappa \mid \kappa^{\prime} $ and $ \kappa^{\prime} \mid \rho_{1} $. 
Therefore $ \phi(\mathfrak{a}) = f_{1}^{\kappa^{\prime}}(N(\mathfrak{a})) = f_{1}^{\kappa}(N(\mathfrak{a})) $. 
Hence $ f^{\kappa}(t) \coloneqq f_{1}^{\kappa}(t) $ satisfies the desired condition. 

(\ref{minimum period 2}) 
Let $\rho_0$ be the greatest common divisor of all periods of $\phi$. 
It suffices to show that $ \rho_{0} $ is a period of $ \phi $. 
Let $ \mathfrak{p}_1, \dots, \mathfrak{p}_s $ be the prime factors of $ \rho_{0} $. 
Then there exists a period $ \rho_{i} $ of $ \phi $ such that $\ord_{\mathfrak{p}_i} (\rho_i) = \ord_{\mathfrak{p}_i} (\rho_0)$ for each $ i \in \{1, \dots, s\} $. 
Since 
\begin{align*}
\rho_0 
= \prod_{i=1}^s \mathfrak{p}_i^{\ord_{\mathfrak{p}_i} (\rho_0)}
= \rho_1 + \cdots + \rho_s. 
\end{align*}
we can conclude that $\rho_0$ is a period of $\phi$ by (\ref{minimum period 1}). 
\end{proof}

\begin{proposition}\label{sum of quasi-polynomials}
Suppose that $ \phi_{1} $ and $ \phi_{2} $ are quasi-polynomials on $ I(\mathcal{O}) $ with period $ \rho_{1} $ and $ \rho_{2} $, respectively. 
Then the function $ (\phi_{1} + \phi_{2})(\mathfrak{a}) \coloneqq \phi_{1}(\mathfrak{a}) + \phi_{2}(\mathfrak{a}) $ is a quasi-polynomial of $ I(\mathcal{O}) $ with period $ \lcm(\rho_{1}, \rho_{2}) $. 
\end{proposition}
\begin{proof}
For $ i \in \{1,2\} $, let $ f_{i}^{\kappa}(t) \in \mathbb{Z}[t] $ denote the $ \kappa $-constituent of $ \phi_{i} $ associated with the period $ \rho_{i} $. 
For any divisor $ \kappa $ of $ \lcm(\rho_{1}, \rho_{2}) $, define $ f^{\kappa}(t) \in \mathbb{Z}[t] $ by $ f^{\kappa}(t) \coloneqq f_{1}^{\kappa + \rho_{1}}(t) + f_{2}^{\kappa + \rho_{2}}(t) $. 
Assume that $ \mathfrak{a} + \lcm(\rho_{1}, \rho_{2}) = \kappa $. 
Then $ \mathfrak{a} + \rho_{i} = \kappa + \rho_{i} $ for each $ i \in \{1,2\} $. 
Therefore $ f^{\kappa}(N(\mathfrak{a})) = f_{1}^{\kappa + \rho_{1}}(N(\mathfrak{a})) + f_{2}^{\kappa + \rho_{2}}(N(\mathfrak{a})) = \phi_{1}(\mathfrak{a}) + \phi_{2}(\mathfrak{a}) = (\phi_{1}+\phi_{2})(\mathfrak{a}) $. 
\end{proof}

\section{The characteristic quasi-polynomial of $\A$} \label{Characteristic quasi-polynomial}

Let $\A = \{ c_1, \dots, c_n \} \subseteq \mathcal{O}^{\ell}$ and $ \mathfrak{a} \in I(\mathcal{O}) $. 
Recall that the \textbf{$\mathfrak{a}$-reduced arrangement} $\A (\mathcal{O} / \mathfrak{a}) \coloneqq \Set{H_{j , \mathfrak{a}} | j \in [n]} $ consists of hyperplanes 
\begin{align*}
H_{j, \mathfrak{a}} \coloneqq \Set{ [ \boldsymbol{x} ]_{\mathfrak{a}} \in \left( \mathcal{O} / \mathfrak{a} \right)^\ell | \boldsymbol{x} c_j \equiv 0 \pmod{\mathfrak{a}} }. 
\end{align*}
Let $ M( \A (\mathcal{O} / \mathfrak{a}) ) $ denote the complement of $ \A (\mathcal{O} / \mathfrak{a}) $. 
Namely
\begin{align*}
M( \A (\mathcal{O} / \mathfrak{a}) ) \coloneqq \left(\mathcal{O}/\mathfrak{a}\right)^{\ell} \setminus 
\bigcup_{j=1}^{n}H_{j, \mathfrak{a}}. 
\end{align*}

The main purpose of this section is to prove that the characteristic quasi-polynomial $ \chi_{\mathcal{A}}^{\mathrm{quasi}}(\mathfrak{a}) = \left|M( \A (\mathcal{O} / \mathfrak{a}) )\right| $ is a quasi-polynomial discussed in Section \ref{Quasi-polynomials with GCD-property}. 
The proof will proceed in the similar way with the proof of Theorem \ref{KTT081}. 
However, since in general $ \mathcal{O} $ is not a principal ideal domain, we cannot use the Smith normal forms and we will use the structure theorem of finite generated modules over a Dedekind domain instead of them.

For a nonempty subset $ J = \{ j_1, \dots, j_k \} \subseteq [n] $, we define 
\begin{align*}
H_{J, \mathfrak{a}} \coloneqq \bigcap_{j \in J} H_{j, \mathfrak{a}} = H_{j_1, \mathfrak{a}} \cap \cdots \cap H_{j_k, \mathfrak{a}}. 
\end{align*}
By the inclusion-exclusion principle, 
\begin{align*}
\chi_{\mathcal{A}}^{\mathrm{quasi}}(\mathfrak{a}) 
&= \left| M(\A (\mathcal{O} / \mathfrak{a})) \right| 
= \left| \left(\mathcal{O}/\mathfrak{a}\right)^{\ell} \right| - \left| \bigcup_{j=1}^{n}H_{j, \mathfrak{a}} \right|
= N (\mathfrak{a})^{\ell} - \sum_{k = 1}^{n} (-1)^{k - 1} \sum_{\substack{J \subseteq [n] \\ |J| = k}} 
\left| \bigcap_{j \in J} H_{j, \mathfrak{a}} \right| \\
&= N(\mathfrak{a})^\ell + \sum_{\varnothing \ne J \subseteq [n]} (-1)^{|J|} \left| H_{J, \mathfrak{a}} \right| . 
\end{align*}


In order to prove that $ \chi_{\mathcal{A}}^{\mathrm{quasi}}(\mathfrak{a}) $ is a quasi-polynomial, it is sufficient to show that the cardinality $ \left| H_{J, \mathfrak{a}} \right| $ is a quasi-polynomial for each $ J $ by Proposition \ref{sum of quasi-polynomials}. 

Let $C_J \coloneqq (c_{j_1}, \dots c_{j_k}) \in \mathrm{Mat}_{\ell \times k} (\mathcal{O})$. 
We define an $\mathcal{O}$-homomorphism $\phi_J \colon \mathcal{O}^\ell \rightarrow \mathcal{O}^{|J|} $ by $\phi_J(\boldsymbol{x}) \coloneqq \boldsymbol{x} C_J$. 
It induces an $\mathcal{O}$-homomorphism
$\phi_{J, \mathfrak{a}} \colon \left(\mathcal{O}/\mathfrak{a}\right)^{\ell} \rightarrow \left(\mathcal{O}/\mathfrak{a}\right)^{|J|}$ 
defined by $\phi_{J, \mathfrak{a}} ([\boldsymbol{x}]_{\mathfrak{a}}) \coloneqq [\phi_J (\boldsymbol{x})]_{\mathfrak{a}}$. 
Note that $ \left| H_{J, \mathfrak{a}} \right| = \left| \ker \phi_{J, \mathfrak{a}} \right|$. 

Since $\coker \phi_J = \mathcal{O}^k / \Ima \phi_J $ is a finitely generated $\mathcal{O}$-module, 
by the structure theorem for finitely generated modules over a Dedekind domain (See for example {\cite[Theorem 5.1]{auslander1974groups}}), we obtain the following key lemma. 

\begin{lemma}\label{structure of coker phi_J}
There exist nonzero ideals $ \mathfrak{d}_{J}, \mathfrak{d}_{J,1}, \mathfrak{d}_{J,2}, \dots, \mathfrak{d}_{J,r(J)} $ such that 
\begin{align*}
\coker \phi_J \simeq \bigoplus_{i = 1}^{r(J)} \mathcal{O} / \mathfrak{d}_{J, i} \oplus \mathfrak{d}_{J} \oplus \mathcal{O}^{|J| - r (J)-1} \qquad (\mathfrak{d}_{J,i} \mid \mathfrak{d}_{J, i+1}, \ \text{ for } 1 \leq i \leq r(J)-1), 
\end{align*}
where $ r(J) \coloneqq \rank(\Ima\phi_{J}) $ and if $ r(J) = |J| $, then we regard $ \mathfrak{d}_{J} \oplus \mathcal{O}^{|J| - r (J)-1} $ as $ \{0\} $. 
\end{lemma}

\begin{lemma}[See also {\cite[Lemma 2.1]{kamiya2008periodicity-joac}}] \label{card of ker}
\begin{align*}
\left| H_{J, \mathfrak{a}} \right| = \left| \ker \phi_{J, \mathfrak{a}} \right| 
= \left( \prod_{i = 1}^{r(J)} N (\mathfrak{a} + \mathfrak{d}_{J, i}) \right) N( \mathfrak{a} )^{\ell - r(J)}. 
\end{align*}
Moreover, $ \left| H_{J, \mathfrak{a}} \right| $ is a quasi-monomial with GCD-property of degree  $\ell -r(J) $ and its minimum period is $\mathfrak{d}_{J, r(J)}$. 
\end{lemma}

\begin{proof}
Since 
$ \left(\mathcal{O}/\mathfrak{a}\right)^{\ell} / \ker \phi_{J, \mathfrak{a}} \simeq \Ima \phi_{J, \mathfrak{a}}$ and 
$ \coker \phi_{J, \mathfrak{a}} = \left(\mathcal{O}/\mathfrak{a}\right)^{k} / \Ima \phi_{J, \mathfrak{a}}$, 
we have 
$ \left| \Ima \phi_{J, \mathfrak{a}} \right| = N(\mathfrak{a})^\ell / \left| \ker \phi_{J, \mathfrak{a}} \right|$ and 
$ \left| \coker \phi_{J, \mathfrak{a}} \right| = N(\mathfrak{a})^k / \left| \Ima \phi_{J, \mathfrak{a}} \right|$. 
Hence we obtain 
\begin{align*}
\left| \ker \phi_{J, \mathfrak{a}} \right| 
= \dfrac{N(\mathfrak{a})^\ell}{\left| \Ima \phi_{J, \mathfrak{a}} \right|} 
= \dfrac{N(\mathfrak{a})^\ell \left| \coker \phi_{J, \mathfrak{a}} \right|}{N(\mathfrak{a})^k} 
= N(\mathfrak{a})^{\ell - k} \left| \coker \phi_{J, \mathfrak{a}} \right|. 
\end{align*}

Consider the exact sequence 
\begin{align*}
\mathcal{O}^\ell \overset{\phi_J}{\longrightarrow} \mathcal{O}^k \longrightarrow \coker \phi_J \longrightarrow 0. 
\end{align*}
Since tensoring with $ \mathcal{O}/\mathfrak{a} $ is a right exact functor, we obtain the exact sequence 
\begin{align*}
\mathcal{O}^\ell \otimes_{\mathcal{O}} \mathcal{O} / \mathfrak{a} 
\overset{\phi_J \otimes \id}{\longrightarrow} 
\mathcal{O}^k \otimes_{\mathcal{O}} \mathcal{O} / \mathfrak{a} 
\longrightarrow 
\left( \coker \phi_J \right) \otimes_{\mathcal{O}} \mathcal{O} / \mathfrak{a} 
\longrightarrow 0. 
\end{align*}

Also, we have the following commutative diagram of $ \mathcal{O} $-modules, where the vertical arrows show isomorphisms: 

\begin{center}
\begin{tikzcd}
  \mathcal{O}^{\ell} \otimes_{\mathcal{O}} \mathcal{O} / \mathfrak{a}   \arrow[r, "\phi_{J} \otimes \id"] \arrow[d] & \mathcal{O}^{k} \otimes_{\mathcal{O}} \mathcal{O} / \mathfrak{a} \arrow[d] \\
  \left( \mathcal{O} / \mathfrak{a} \right)^{\ell} \arrow[r, "\phi_{J, \mathfrak{a}}"] & \left( \mathcal{O} / \mathfrak{a} \right)^{k}
\end{tikzcd}
\end{center}

Then we obtain 
\begin{align*}
\coker \phi_{J, \mathfrak{a}} 
&\simeq 
\left( \coker \phi_J \right) \otimes_{\mathcal{O}} \mathcal{O} / \mathfrak{a} \\
&\simeq
\bigoplus_{i = 1}^{r(J)} \left( \mathcal{O} / \mathfrak{d}_{J, i} \otimes_{\mathcal{O}} \mathcal{O} / \mathfrak{a} \right) 
\oplus \left( \mathfrak{d}_{J} \otimes_{\mathcal{O}} \mathcal{O} / \mathfrak{a} \right) 
\oplus \left( \mathcal{O}^{k - r(J) -1} \otimes_{\mathcal{O}} \mathcal{O} / \mathfrak{a} \right) 
\\
&\simeq 
\bigoplus_{i = 1}^{r(J)} \mathcal{O} / \left( \mathfrak{a} + \mathfrak{d}_{J, i} \right)  
\oplus \mathfrak{d}_{J}/ \mathfrak{a} \mathfrak{d}_{J} 
\oplus \left( \mathcal{O} / \mathfrak{a} \right)^{k - r(J) - 1}. 
\end{align*}
Since $\mathcal{O}$ is a Dedekind domain, $\mathfrak{d}_{J}/ \mathfrak{a} \mathfrak{d}_{J} \simeq \mathcal{O} / \mathfrak{a}$. 
Hence 
\begin{align*}
\coker \phi_{J, \mathfrak{a}} 
\simeq 
\bigoplus_{i = 1}^{r(J)} \mathcal{O} / \left( \mathfrak{a} + \mathfrak{d}_{J, i} \right)  
\oplus \left( \mathcal{O} / \mathfrak{a} \right)^{k - r(J)}. 
\end{align*}
Therefore we obtain 
\begin{align*}
\left| H_{J, \mathfrak{a}} \right| &= \left| \ker \phi_{J, \mathfrak{a}} \right| 
= N(\mathfrak{a})^{\ell - k} \left| \coker \phi_{J, \mathfrak{a}} \right| 
= N(\mathfrak{a})^{\ell - k} \cdot \left( \prod_{i = 1}^{r(J)} N \left( \mathfrak{a} + \mathfrak{d}_{J, i} \right) \right) \cdot N(\mathfrak{a})^{k - r(J)} \\
&= \left( \prod_{i = 1}^{r(J)} N \left( \mathfrak{a} + \mathfrak{d}_{J, i} \right) \right) N(\mathfrak{a})^{\ell - r(J)}. 
\end{align*}

Next, we will show that $ \left| H_{J, \mathfrak{a}} \right| $ is a quasi-monomial with period $ \mathfrak{d}_{J, r(J)} $. 
For each $ \kappa \mid \mathfrak{d}_{J, r(J)} $, define the monomial $ f^{\kappa}(t) \in \mathbb{Z}[t] $ by 
\begin{align*}
f^{\kappa}(t) \coloneqq \left( \prod_{i = 1}^{r(J)} N \left( \kappa + \mathfrak{d}_{J, i} \right) \right)t^{\ell-r(J)}. 
\end{align*}
Suppose that $ \mathfrak{a} + \mathfrak{d}_{J,r(J)} = \kappa $. 
Then $ \mathfrak{a} + \mathfrak{d}_{J,i} = \kappa + \mathfrak{d}_{J,i} $ for each $ i \in \{1, \dots, r(J)\} $ since $ \mathfrak{d}_{J,i} \mid \mathfrak{d}_{J,r(J)} $. 
Hence $ f^{\kappa}(N(\mathfrak{a})) = \left| H_{J, \mathfrak{a}} \right| $. 

Finally, we will show that $\mathfrak{d}_{J, r(J)}$ is the minimum period of $ \left| H_{J, \mathfrak{a}} \right|$. 
Let $\rho_{0}$ be the minimum period. 
For any $ \kappa \mid \rho_{0} $ there exists $ g^{\kappa}(t) \in \mathbb{Z}[t] $ such that $ \mathfrak{a} + \rho_{0} = \kappa $ implies $ g^{\kappa}(N(\mathfrak{a})) = |H_{J, \mathfrak{a}}| $. 
Suppose that $ \mathfrak{a} $ is a multiple of $ \mathfrak{d}_{J,r(J)} $. 
Then $ \mathfrak{a} + \mathfrak{d}_{J,r(J)} = \mathfrak{d}_{J, r(J)} $ and $ \mathfrak{a} + \rho_{0} = \rho_{0} $.
Therefore $ f^{\mathfrak{d}_{J,r(J)}}(N(\mathfrak{a})) = |H_{J,\mathfrak{a}}| = g^{\rho_{0}}(N(\mathfrak{a})) $. 
Since there are infinitely many such ideals $ \mathfrak{a} $, we have $ f^{\mathfrak{d}_{J,r(J)}}(t) = g^{\rho_{0}}(t) $. 
Hence 
\begin{align*}
f^{\mathfrak{d}_{J,r(J)}}(N(\rho_{0})) = g^{\rho_{0}}(N(\rho_{0})) = |H_{J,\rho_{0}}| = f^{\rho_{0}}(N(\rho_{0})). 
\end{align*}
Therefore 
\begin{align*}
\prod_{i=1}^{r(J)}N(\mathfrak{d}_{J,i}) = \prod_{i=1}^{r(J)}N(\mathfrak{d}_{J,r(J)} + \mathfrak{d}_{J,i}) = \prod_{i=1}^{r(J)}N(\rho_{0} + \mathfrak{d}_{J,i}). 
\end{align*}
Since $ N(\mathfrak{d}_{J,i}) \geq N(\mathfrak{d}_{J,i} + \rho_{0})  $ for each $ i \in \{ 1, \dots, r(J) \} $, we have $ N(\mathfrak{d}_{J,i}) = N(\mathfrak{d}_{J,i} + \rho_{0}) $ for every $ i \in \{ 1, \dots, r(J) \} $. 
In particular $ N(\mathfrak{d}_{J,r(J)}) = N(\mathfrak{d}_{J,r(J)} + \rho_{0}) = N(\rho_{0}) $. 
Since $ \rho_{0} \mid \mathfrak{d}_{J,r(J)} $, we have $ \rho_{0} = \mathfrak{d}_{J,r(J)} $. 
\end{proof}

From the discussion so far, we have 
\begin{align*}
\chi_{\mathcal{A}}^{\mathrm{quasi}}(\mathfrak{a}) 
&= N(\mathfrak{a})^\ell + \sum_{\varnothing \ne J \subseteq [n]} (-1)^{|J|} \left| H_{J, \mathfrak{a}} \right| \\
&= N(\mathfrak{a})^\ell + \sum_{\varnothing \ne J \subseteq [n]} (-1)^{|J|} \left( \prod_{i = 1}^{r(J)} N \left( \mathfrak{a} + \mathfrak{d}_{J, i} \right) \right) N(\mathfrak{a})^{\ell - r(J)}. 
\end{align*}
Since $\mathfrak{d}_{J, r(J)}$ is a period of $\left| H_{J, \mathfrak{a}} \right|$, 
\begin{align*}
\rho_{\mathcal{A}} \coloneqq \lcm \Set{\mathfrak{d}_{J, r(J)} | \varnothing \ne J \subseteq [n]}
\end{align*}
is a period of $ \chi_{\mathcal{A}}^{\mathrm{quasi}}(\mathfrak{a})  $ by Proposition \ref{sum of quasi-polynomials}. 
We call $ \rho_{\mathcal{A}} $ the \textbf{LCM-period}. 
In summary, we obtain the following results:

\begin{theorem}[Generalization of Theorem \ref{KTT081}] \label{main thm}
The characteristic quasi-polynomial $ \chi_{\mathcal{A}}^{\mathrm{quasi}}(\mathfrak{a})  $ is a monic quasi-polynomial on $ I(\mathcal{O}) $ of degree $ \ell $ with period $ \rho_{\mathcal{A}} $ and for each $ \kappa \mid \rho_{\mathcal{A}} $ the $ \kappa $-constituent is given by 
\begin{align*}
f_{\A}^{\kappa} (t) = \sum_{J \subseteq [n]} (-1)^{|J|} \, m(J, \kappa) t^{\ell - r(J)}, 
\end{align*}
where we put $m(J, \mathfrak{a}) \coloneqq |\mathrm{tors}(\coker(\phi_{J, \mathfrak{a}}))| = \prod_{i = 1}^{r(J)} N \left( \mathfrak{a} + \mathfrak{d}_{J, i} \right)$ for nonempty $J \subseteq [n]$ and $ \mathfrak{a} \in I(\mathcal{O}) $. 
When $J = \varnothing$, we define $r(\varnothing) \coloneqq 0$ and $m(\varnothing, \kappa) \coloneqq 1$. 
\end{theorem}


\begin{remark}
The LCM-period $\rho_{\mathcal{A}}$ is actually the minimum period of $\chi^{\mathrm{quasi}}_{\A}$ (See Theorem \ref{main thm3}).
\end{remark}

\begin{theorem}[Generalization of Theorem \ref{KTT082}] \label{main thm1.5}
The $\langle 1 \rangle$-constituent of the characteristic quasi-polynomial of $\A$ coincides with the characteristic polynomial of the hyperplane arrangement $\A (K)$, that is, $f_{\A}^{\langle 1 \rangle} (t) = \chi_{\A(K)} (t)$. 
\end{theorem}

\begin{proof}
By Theorem \ref{main thm}, 
\begin{align*}
f_{\A}^{\langle 1 \rangle} (t) 
= \sum_{J \subseteq [n]} (-1)^{|J|}  t^{\ell - r(J)}. 
\end{align*}
Using Whitney's theorem 
(See for example \cite[Theorem 2.4]{stanley2004introduction-lnicmi}), we obtain 
$\chi_{\A (K)} (t) = f_{\A}^{\langle 1 \rangle} (t) $. 
\end{proof}

\begin{example} \label{example2}
Let $ \mathcal{O} = \mathbb{Z}[\sqrt{-1}] $ and $ \mathcal{A} = \Set{\begin{pmatrix}
1 \\ 1
\end{pmatrix}, \begin{pmatrix}
1 \\ -1
\end{pmatrix}, \begin{pmatrix}
1 \\ \sqrt{-1}
\end{pmatrix}, \begin{pmatrix}
1 \\ -\sqrt{-1}
\end{pmatrix}} $. 
Since $ \mathcal{O} $ is a principal ideal domain, we can determine the LCM-period by computing Smith normal forms. 
The LCM-period is $ \langle 2 \rangle = \mathfrak{p}^{2} $, where $ \mathfrak{p} = \langle 1+\sqrt{-1} \rangle $. 
The constituents of $ \chi_{\mathcal{A}}^{\mathrm{quasi}} $ are given by 
\begin{align*}
f_{\mathcal{A}}^{\langle 1 \rangle} = t^{2}-4t + 3, \quad
f_{\mathcal{A}}^{\mathfrak{p}} = t^{2}-4t + 6, \quad
f_{\mathcal{A}}^{\langle 2 \rangle} = t^{2}-4t + 10.  
\end{align*}
\end{example}

\begin{example} \label{example1}
Let $K$ be a quadratic field $\mathbb{Q} (\sqrt{-5})$. 
Then its ring of integers is $\mathcal{O} = \mathbb{Z} [\sqrt{-5}]$. 
Note that $ \mathcal{O} $ is not a principal ideal domain. 
For example, prime ideals $ \mathfrak{p} \coloneqq \langle 2, 1-\sqrt{-5}\rangle $ and $ \mathfrak{q} \coloneqq \langle 3, 1+\sqrt{-5} \rangle $ are not principal. 

Let 
\begin{align*}
\A \coloneqq \{ c_1, c_2 \}, 
\ \text{ where }
c_{1} \coloneqq \begin{pmatrix}
2  \\
1 - \sqrt{-5}
\end{pmatrix}, \
c_{2} \coloneqq \begin{pmatrix}
1 + \sqrt{-5} \\
3
\end{pmatrix}. 
\end{align*}

Then we have 
\begin{align*}
\Ima \phi_{\{1\}} = \mathfrak{p}, \quad
\Ima \phi_{\{2\}} = \mathfrak{q}, \quad
\Ima \phi_{\{1,2\}} = \left\langle (2,1+\sqrt{-5}), \ (1-\sqrt{-5}, 3) \right\rangle. 
\end{align*}

Define a map $ g \colon \mathcal{O}^{2} \to \mathcal{O} $ by $ g(a,b) \coloneqq (1+\sqrt{-5})a-2b $. 
Then it holds that $ \ker g = \Ima\phi_{\{1,2\}} $ and $ \Ima g = \mathfrak{p} $. 
We will give a proof only for $ \ker g \subseteq \Ima\phi_{\{1,2\}} $ since the rest is clear. 
Taking $ (a,b) \in \ker g $, we have $ 2b = (1+\sqrt{-5})a \in \langle 1+\sqrt{-5} \rangle = \mathfrak{p}\mathfrak{q} \subseteq \mathfrak{q} $. 
Since $ 2 \not\in \mathfrak{q} $, we have $ b \in \mathfrak{q} $. 
Therefore there exist $ c,d \in \mathcal{O} $ such that $ b = 3c + (1+\sqrt{-5})d $. 
Then $ (1+\sqrt{-5})a = 2b = 6c + 2(1+\sqrt{-5})d $ and hence $ a = (1-\sqrt{-5})c + 2d $. 
Hence $ (a,b) = ((1-\sqrt{-5}c+2d, 3c+(1+\sqrt{-5})d)) = (1-\sqrt{-5},3)c + (2,1+\sqrt{-5})d \in \Ima\phi_{\{1,2\}} $. 
Thus, $ \Ima\phi_{\{1,2\}} \supseteq \ker g $. 

We obtain 
\begin{align*}
\coker\phi_{\{1\}} = \mathcal{O}/\mathfrak{p}, \quad
\coker\phi_{\{2\}} = \mathcal{O}/\mathfrak{q}, \quad
\coker\phi_{\{1,2\}} = \Ima g = \mathfrak{p}. 
\end{align*}
Therefore $\rho_{\mathcal{A}} = \operatorname{lcm} \{ \mathfrak{p}, \mathfrak{q} \} = \mathfrak{p} \mathfrak{q} = \langle 1+\sqrt{-5} \rangle$ and the constituents of $\chi_{\A}^{\mathrm{quasi}}$ are given by 
\begin{align*}
f^{\langle 1 \rangle}_\A (t) 
= t^2 - t, 
\quad
f^{\mathfrak{p}}_\A (t) 
= t^2 - 2t, 
\quad
f^{\mathfrak{q}}_\A (t) 
= t^2 - 3t, 
\quad
f^{\mathfrak{p}\mathfrak{q}}_\A (t) 
= t^2 - 4t.
\end{align*}

Moreover, since $ (1+\sqrt{-5})c_{1} = 2c_{2} $, the hyperplanes $ H_{1} = \Set{\boldsymbol{x} \in K^{2} | \boldsymbol{x}c_{1}=0} $ and $ H_{2} = \Set{\boldsymbol{x} \in K^{2} | \boldsymbol{x}c_{2}=0} $ are equal. 
Therefore $ \mathcal{A}(K) = \{H_{1}\} $ and we obtain $ \chi_{\mathcal{A}(K)}(t) = t^{2}-t = f_{\mathcal{A}}^{\langle 1 \rangle}(t) $ directly. 
\end{example}

\section{The torsion arrangement and its poset of layers} \label{torsion arrangement}

Recall that the hyperplane arrangement $ \mathcal{A}(K) $ is defined by $ \mathcal{A}(K)= \{H_{1}, \dots, H_{n}\} $, where 
\begin{align*}
H_{j} = \Set{\boldsymbol{x} \in K^{\ell} | \boldsymbol{x}c_{j} = 0}
\end{align*}
and the torsion arrangement $ \A (K / \mathcal{O}) $ is defined by $\A (K / \mathcal{O}) = \Set{T_{1}, \dots, T_{n} }$, where 
\begin{align*}
T_j = \Set{ \pi ( \boldsymbol{x} ) \in (K/\mathcal{O})^{\ell} | \boldsymbol{x} c_j \equiv 0 \pmod{\mathcal{O}} }
\end{align*}
and $\pi \colon K^{\ell} \to (K/\mathcal{O})^{\ell}$ denotes the natural projection. 

For each $\varnothing \ne J \subseteq [n]$, let $ H_{J} \coloneqq \bigcap_{j \in J} H_{j} $ and $T_J \coloneqq \bigcap_{j \in J} T_j$. 
Moreover, let $ H_{\varnothing} \coloneqq K^{\ell} $ and $T_{\varnothing} \coloneqq (K/\mathcal{O})^{\ell} $. 
Note that for each $\varnothing \ne J \subseteq [n]$, 
\begin{align*}
\pi(H_{J}) &= \Set{\pi(\boldsymbol{x}) \in (K/\mathcal{O})^{\ell} | \boldsymbol{x} C_J = 0 }, \\
T_{J}  &= \Set{ \pi (\boldsymbol{x}) \in (K/\mathcal{O})^{\ell} | \boldsymbol{x} C_J \in \mathcal{O}^{|J|} }
\end{align*}
and $ \pi(H_{J}) $ is an $ \mathcal{O} $-submodule of $ T_{J} $. 

The intersection poset of $ \mathcal{A}(K) $ is defined by 
\begin{align*}
L(\mathcal{A}(K)) \coloneqq \Set{H_{J} \subseteq K^{\ell} | J \subseteq [n]} 
\end{align*}
equipped with the order defined by the reverse inclusion. 
It is well known that $ L(\mathcal{A}) $ is a finite geometric lattice. 

To construct the intersection poset for toric arrangements, one uses connected components of the intersections of tori. 
However, the field $ K $ may not have reasonable topology to define connected components of the intersection $ T_{J} $. 
To define the poset for the torsion arrangement $ \mathcal{A}(K/\mathcal{O}) $, we use the quotient module $ T_{J}/\pi(H_{J}) $ (See also \cite[The paragraph before Lemma 4.5]{tran2019combinatorics-joctsa}).

\begin{definition}
Define the \textbf{poset of layers} $ L(\A (K / \mathcal{O}) ) $ by 
\begin{align*}
L(\A (K / \mathcal{O}) ) \coloneqq & \Set{ Z \subseteq \left( K / \mathcal{O} \right)^{\ell} | 
J \subseteq [n], \ Z \in T_{J} / \pi (H_J) } \\
= & \Set{\pi(\boldsymbol{x}) + \pi(H_{J}) | J \subseteq [n], \boldsymbol{x} \in T_{J}}. 
\end{align*}
equipped with the order defined by the reverse inclusion. 
The poset $ L(\A (K / \mathcal{O}) ) $ has a unique minimal element $T_{\varnothing} = (K/\mathcal{O})^{\ell}$. 
We call an element of $L(\A (K / \mathcal{O}) ) $ a \textbf{layer}.
\end{definition}

\begin{remark}
Every $ T_{J} $ can be represented as the disjoint union of layers: $ T_{J} = \bigsqcup_{Z \in T_{J}/\pi(H_{J})} Z $. 
Thus a layer behaves like a ``connected component" of the intersection $ T_{J} $. 
For each $ J $, $ \pi(H_{J}) $ is a unique component of $ T_{J} $ containing the unit element $ \pi(\boldsymbol{0}) $. 
\end{remark}

\subsection{Finiteness of the poset of layers}

\begin{lemma} \label{hissu}
For each $\varnothing \ne J \subseteq [n]$, 
\begin{align*} 
T_J / \pi (H_J) \simeq 
\mathrm{tors} \left(\coker \phi_J \right) 
\simeq \bigoplus_{i = 1}^{r(J)} \mathcal{O} / \mathfrak{d}_{J, i}. 
\end{align*}
as $\mathcal{O}$-modules, 
where $\phi_J \colon \mathcal{O}^{\ell} \to \mathcal{O}^{|J|}$ is defined by $\phi_J (\boldsymbol{x}) = \boldsymbol{x} C_J$ and $ \mathfrak{d}_{J, 1}, \dots, \mathfrak{d}_{J, r(J)} $ are the ideals  determined in Lemma \ref{structure of coker phi_J}. 
\end{lemma}

\begin{proof}
The latter isomorphism follows by Lemma \ref{structure of coker phi_J}. 
We verify the former isomorphism. 

Define an $\mathcal{O}$-homomorphism $ f_{J} \colon T_{J}/\pi(H_{J}) \to \mathrm{tors}(\coker \phi_{J}) $ by $ f_{J}(\boldsymbol{x} + \pi(H_{J})) \coloneqq \boldsymbol{x}C_{J} + \Ima \phi_{J} $. 
Note that when $ \boldsymbol{x} \in T_{J} $ we have $ \boldsymbol{x}C_{J} \in \mathcal{O}^{J} $ and, since every element in $ (K/\mathcal{O})^{\ell} $ is a torsion element, we have $ \boldsymbol{x}C_{J} + \Ima\phi_{J} \in \mathrm{tors}(\coker\phi_{J}) $. 

First, we prove that $ f_{J} $ is well-defined. 
Suppose that $\pi (\boldsymbol{x}) + \pi (H_J) = \pi ( \boldsymbol{y} ) + \pi (H_J) \in T_J / \pi (H_J) $. 
Then $\pi (\boldsymbol{x} ) - \pi (\boldsymbol{y}) \in \pi (H_J)$. 
Hence there exists $\boldsymbol{v} \in H_J$ such that $ \pi (\boldsymbol{x} ) - \pi (\boldsymbol{y}) = \pi(\boldsymbol{v})$, 
and hence there exists $\boldsymbol{a} \in \mathcal{O}^{\ell}$ such that $\boldsymbol{x} - \boldsymbol{y} - \boldsymbol{v} = \boldsymbol{a}$. 
Since $\boldsymbol{v} \in H_J$ and $\boldsymbol{a} \in \mathcal{O}^{\ell} $, 
$\boldsymbol{x} C_J - \boldsymbol{y} C_J = \boldsymbol{v} C_J + \boldsymbol{a} C_J = \boldsymbol{a} C_J \in \Ima \phi_J$. 
Thus $ \boldsymbol{x}C_{J} + \Ima\phi_{J} = \boldsymbol{y}C_{J} + \Ima\phi_{J} $ and hence $ f_{J} $ is well-defined. 

Second, we prove that $ f_{J} $ is injective. 
Assume that $ \boldsymbol{x} C_J + \Ima \phi_J = \boldsymbol{y} C_J + \Ima \phi_J  $, where $ \boldsymbol{x}, \boldsymbol{y} \in K^{\ell} $. 
Then 
$\boldsymbol{x} C_J - \boldsymbol{y} C_J \in \Ima \phi_J$ and hence there exists $\boldsymbol{a} \in \mathcal{O}^{\ell}$ such that 
$\boldsymbol{x} C_J - \boldsymbol{y} C_J = \phi_J (\boldsymbol{a}) = \boldsymbol{a} C_J$. 
Therefore $(\boldsymbol{x} - \boldsymbol{y} - \boldsymbol{a}) C_J = 0$ and $\boldsymbol{x} - \boldsymbol{y} - \boldsymbol{a} \in H_J$. 
Thus $\pi (\boldsymbol{x} ) - \pi (\boldsymbol{y}) = \pi (\boldsymbol{x} ) - \pi (\boldsymbol{y}) - \pi (\boldsymbol{a} ) = \pi(\boldsymbol{x}-\boldsymbol{y}-\boldsymbol{a}) \in \pi (H_J)$ and hence $ f_{J} $ is injective. 

Finally, we prove that $ f $ is surjective. 
For any element $\boldsymbol{y} + \Ima \phi_J \in \mathrm{tors} \left( \coker\phi_{J} \right)$, 
there exists $b \in \mathcal{O} \setminus \{ 0 \}$ such that $b \left( \boldsymbol{y} + \Ima \phi_J \right) = \Ima \phi_J$, that is, 
$b \boldsymbol{y} \in \Ima \phi_J$. 
Hence there exists $\boldsymbol{a} \in \mathcal{O}^{\ell}$ such that 
$b \boldsymbol{y} = \phi_J (\boldsymbol{a}) = \boldsymbol{a} C_J$. 
Let $\boldsymbol{x} = (1/b) \boldsymbol{a} \in K^{\ell}$. 
Then $\boldsymbol{x} C_J = (1/b) \boldsymbol{a} C_J = \boldsymbol{y} \in \mathcal{O}^{|J|}$, 
and hence $\pi (\boldsymbol{x}) \in T_J$. 
Therefore
$f_J \left( \pi (\boldsymbol{x}) + \pi (H_J) \right) = \boldsymbol{x} C_J + \Ima \phi_J = \boldsymbol{y} + \Ima \phi_J$. 
Hence $f_J$ is surjective. 

Thus $ f_{J} $ is isomorphic and we obtain the desired isomorphism $T_{J} / \pi(H_J) \simeq \mathrm{tors} \left(\coker \phi_J \right) $ as $\mathcal{O}$-modules. 
\end{proof}

\begin{theorem}
The poset of layers $ L(\mathcal{A}(K/\mathcal{O})) $ is finite. 
\end{theorem}
\begin{proof}
It follows immediately from Lemma \ref{hissu} and the assumption that $ \mathcal{O} $ is residually finite. 
Also note that $T_{\varnothing} / \pi (H_{\varnothing}) = \{ T_{\varnothing} \}$. 
\end{proof}

\subsection{Relations between $ L(\mathcal{A}(K/\mathcal{O})) $ and $ L(\mathcal{A}(K)) $}

First of all, we will show a crucial lemma for describing relations between $ L(\mathcal{A}(K/\mathcal{O})) $ and $ L(\mathcal{A}(K) $. 
We need the topology on $ K^{\ell} $ defined by a fixed nonzero prime ideal $ \mathfrak{p} $ of $ \mathcal{O} $. 
Let $\mathcal{O}_{\mathfrak{p}}$ be the local ring of $\mathcal{O}$ at ${\mathfrak{p}}$. 
Then $\mathcal{O}_{\mathfrak{p}}$ is a discrete valuation ring with the uniformizer $ p $.  
Since $K = \Frac \left( \mathcal{O}_{\mathfrak{p}} \right)$, 
any element $x \in K \setminus \{ 0 \}$ is given by the form 
$x = u p^{m}$ uniquely, where $u \in \mathcal{O}_{\mathfrak{p}}^{\times}$ and $m \in \mathbb{Z}$. 
We denote the exponent $m$ by $\ord _p (x)$. 
Putting $\ord_p (0) \coloneqq \infty$, we obtain a discrete valuation $\ord_p \colon K \to \mathbb{Z} \cup \{ \infty \}$ and then $K$ is a topological space with an open base consisting of 
$
B_{m} (a) \coloneqq 
\Set{x \in K | \ord_p (x - a) > m} $, where $a \in K$ and $m \in \mathbb{Z}$. 

We can consider that $K^{\ell}$ is a topological space defined by the product topology. 
Note that any vector subspace of $K^{\ell}$ is closed and $K^{\ell}$ has an open base consisting of 
\begin{align*}
B_{m} (\boldsymbol{a}) \coloneqq 
\Set{\boldsymbol{x} \in K^{\ell} | \ord_p (\boldsymbol{x} - \boldsymbol{a} ) > m}, 
\end{align*} 
where $ \boldsymbol{x} = (x_{1}, \dots, x_{\ell}), \, \boldsymbol{a} = (a_{1}, \dots, a_{\ell}) \in K^{\ell}, \, m \in \mathbb{Z}$, and $ \ord_p (\boldsymbol{x} - \boldsymbol{a} ) \coloneqq \min_{1 \leq i \leq \ell} \ord_p (x_i - a_i) $. 

\begin{lemma}\label{locally dvr}
Let $ U $ and $ V $ be vector subspaces of $ K^{\ell} $ and $ \boldsymbol{x}_{1}, \dots, \boldsymbol{x}_{r} \in K^{\ell} $. 
Suppose that $ \pi(U) \subseteq \bigcup_{i=1}^{r}(\pi(\boldsymbol{x}_{i}) + \pi(V) ) $. 
Then $ U \subseteq V $. 
\end{lemma}
\begin{proof}
Let $ \boldsymbol{u} \in U $. 
We will show that $ \boldsymbol{u} $ is an accumulation point of $ V $, that is, for any $ m \in \mathbb{Z}, B_{m}(\boldsymbol{u}) \cap V \neq \varnothing $. 
Take an element $ a \in \mathcal{O}\setminus\{0\} $ such that $ a\boldsymbol{x}_{i} \in \mathcal{O}^{\ell} $ for each $ i \in \{1, \dots, r\} $ and consider the vector $ a^{-1}p^{-m-1}\boldsymbol{u} \in U $. 

From the assumption, we have $ \pi(a^{-1}p^{-m-1}\boldsymbol{u}) = \pi(\boldsymbol{x}_{i}) + \pi(\boldsymbol{v}) $ for some $ \boldsymbol{x}_{i} $ and $ \boldsymbol{v} \in V $. 
Then $ \pi(p^{-m-1}\boldsymbol{u}) = \pi(a\boldsymbol{x}_{i}) + \pi(a\boldsymbol{v}) = \pi(\boldsymbol{0}) + \pi(a\boldsymbol{v}) = \pi(a\boldsymbol{v}) $. 
Therefore there exists $ \boldsymbol{b} \in \mathcal{O}^{\ell} $ such that $ p^{-m-1}\boldsymbol{u} = a\boldsymbol{v} + \boldsymbol{b} $. 
Then $ \ord_{p}(ap^{m+1}\boldsymbol{v} - \boldsymbol{u}) = \ord_{p}(p^{m+1}\boldsymbol{b}) > m $.
Thus $ ap^{m+1}\boldsymbol{v} \in B_{m}(\boldsymbol{u}) \cap V $ and hence $ \boldsymbol{u} $ is an accumulation point of $ V $. 

Since the vector subspace $ V $ is closed, we have $ \boldsymbol{u} \in V $. 
Hence $ U \subseteq V $.  
\end{proof}

\begin{proposition}\label{pi preserves order}
Define a map $ \pi \colon L(\mathcal{A}(K)) \to  L(\A (K / \mathcal{O}) ) $ by $ H_{J} \mapsto \pi(H_{J}) $. 
Then $ \pi $ is an order-preserving injection. 
\end{proposition}
\begin{proof}
It is clear that $ \pi $ preserves the order. 
Injectivity follows from Lemma \ref{locally dvr}. 
\end{proof}

\begin{proposition}\label{layer connectedness}
Let $ Z \in T_{J_{1}}/\pi(H_{J_{1}}) $ be a layer. 
Suppose that $ Z \subseteq T_{J_{2}} $. 
Then there exists a layer $ W \in T_{J_{2}}/\pi(H_{J_{2}}) $ such that $ Z \subseteq W $. 
\end{proposition}
\begin{proof}
Let $ \pi(\boldsymbol{x}) \in Z $. 
Then $ Z = \pi(\boldsymbol{x}) + \pi(H_{J_{1}}) $. 
Since $ T_{J_{2}}/\pi(H_{J_{2}}) $ is finite, there exist $ \pi(\boldsymbol{x}_{1}), \dots, \pi(\boldsymbol{x}_{r}) \in T_{J_{2}} $ such that 
\begin{align*}
\pi(\boldsymbol{x}) + \pi(H_{J_{1}}) = Z \subseteq T_{J_{2}} = \bigsqcup_{i=1}^{r}(\pi(\boldsymbol{x}_{i}) + \pi(H_{J_{2}})). 
\end{align*}
Then $ \pi(H_{J_{1}}) \subseteq \bigsqcup_{i=1}^{r} (\pi(\boldsymbol{x}_{i}-\boldsymbol{x}) + \pi(H_{J_{2}})) $. 
Using Lemma \ref{locally dvr}, we have $ H_{J_{1}} \subseteq H_{J_{2}} $. 
Therefore the layer $ W \coloneqq \pi(\boldsymbol{x}) + \pi(H_{J_{2}}) $ is a desired layer. 
\end{proof}

\begin{lemma} \label{intersection}
Given a layer $ Z \in L(\mathcal{A}(K/\mathcal{O})) $, there exists a unique $ H_{Z} \in L(\mathcal{A}) $ such that $ Z = \pi(\boldsymbol{x}) + \pi(H_{Z}) $ for any $ \pi(\boldsymbol{x}) \in Z $. 
\end{lemma}
\begin{proof}
Since $ Z \in T_{J}/\pi(H_{J}) $ for some $ J \subseteq [n] $, the existence holds. 
To show the uniqueness, suppose that $ Z = \pi(\boldsymbol{x}) + \pi(H_{J_{1}}) = \pi(\boldsymbol{x}) + \pi(H_{J_{2}}) $, where $ \boldsymbol{x} \in Z $ and $ J_{i} \subseteq [n] $ for each $ i \in \{1,2\} $. 
Then $\pi (H_{J_1}) = \pi (H_{J_2})$. 
By Lemma \ref{locally dvr}, we obtain $ H_{J_{1}} = H_{J_{2}} $. 
\end{proof}

\begin{definition}
Define a map $ \psi \colon L(\mathcal{A}(K/\mathcal{O})) \to L(\mathcal{A}(K)) $ by $ \psi(Z) \coloneqq H_{Z} $. 
Also, define the \textbf{dimension} of $ Z $ by $ \dim Z \coloneqq \dim H_{Z} $. 
\end{definition}

\begin{definition}
Let $ Z \in L(\mathcal{A}(K/\mathcal{O})) $. 
Define 
\begin{align*}
J_{Z} &\coloneqq \Set{j \in [n] | T_{j} \supseteq Z}, \\
\mathcal{A}_{Z} &\coloneqq \Set{c_{j} \in \mathcal{A} | T_{j} \supseteq Z} = \Set{c_{j} \in \mathcal{A} | j \in J_{Z}}. 
\end{align*}
We call $ \mathcal{A}_{Z} $ the \textbf{localization} of $ \mathcal{A} $ at $ Z $. 
\end{definition}

\begin{lemma} \label{interval}
The following statements hold: 
\begin{enumerate}[(i)]
\item\label{interval 1} 
The map $ \psi $ is order-preserving. 
Namely, if $ Z \subseteq W $, then $ H_{Z} \subseteq H_{W} $. 

\item\label{interval 2} 
The map $ \psi $ induces an isomorphism from the interval $ [T_{\varnothing}, Z] $ to $ L(\mathcal{A}_{Z}(K)) $, where $ [T_{\varnothing}, Z]  \coloneqq \Set{W \in L(\mathcal{A}(K/\mathcal{O})) | T_{\varnothing} \leq W \leq Z} $. 

\item\label{interval 3}  
The M\"{o}bius function on $ L(\mathcal{A}(K/\mathcal{O})) $ strictly alternates in sign. 
\end{enumerate}
\end{lemma}
\begin{proof}
(\ref{interval 1})
Let $ \pi(\boldsymbol{x}) \in Z $. 
Then $ Z = \pi(\boldsymbol{x}) + \pi(H_{Z}) $ and $ W = \pi(\boldsymbol{x}) + \pi(H_{W}) $ since $ Z \subseteq W $.
Then we have $ \pi(H_{Z}) \subseteq \pi(H_{W}) $. 
By Lemma \ref{locally dvr}, $ H_{Z} \subseteq H_{W} $. 

(\ref{interval 2})
We show the injectivity of the restriction $ \psi|_{[T_{\varnothing},Z]} $. 
Let $ W_{1}, W_{2} \in \left[ T_{\varnothing}, Z \right] $. 
Then $ Z \subseteq W_{1} \cap W_{2} $. 
Letting $ \pi(\boldsymbol{x}) \in Z $, we have $ W_{1} = \pi(\boldsymbol{x}) + \pi(H_{W_{1}}) $ and $ W_{2} = \pi(\boldsymbol{x}) + \pi(H_{W_{2}}) $.  
Suppose that $ \psi(W_{1}) = \psi(W_{2}) $. 
Then $ H_{W_{1}} = H_{W_{2}} $ and hence $ W_{1} = W_{2} $. 
Thus the restriction $ \psi|_{[T_{\varnothing},Z]} $ is injective. 

Now, we show the surjectivity. 
Let $ H_{J} \in L(\mathcal{A}_{Z}(K)) $, where $ J \subseteq J_{Z} $. 
Then $ T_{J} \supseteq Z $ and hence there exists a layer $ W \in T_{J}/\pi(H_{J}) $ such that $ W \supseteq Z $ by Proposition \ref{layer connectedness}. 
Then $ \psi(W) = H_{J} $ and hence $ \psi|_{[T_{\varnothing},Z]} $ is surjective. 

(\ref{interval 3})
Note that the restriction of the M\"{o}bius function $ \mu $ on $ L(\mathcal{A}(K/\mathcal{O})) $ to the interval $ [T_{\varnothing}, Z] $ coincides with the M\"{o}bius function on $ [T_{\varnothing}, Z] $. 
By (\ref{interval 2}), the interval $ [T_{\varnothing}, Z] $ is a geometric lattice. 
Therefore the assertion holds. 
\end{proof}

\subsection{The $ \kappa $-torsion subposet}

Let $M$ be a $\mathcal{O}$-module and let $\kappa \in I(\mathcal{O})$. 
An element $m \in M$ is called a \textbf{$\kappa$-torsion element} if 
$\kappa \subseteq \Ann m \coloneqq \Set{a \in \mathcal{O} | am = 0 \in M}$. 

A layer $ \pi(\boldsymbol{x}) + \pi(H_{J}) \in T_{J}/\pi(H_{J}) $ is $ \kappa $-torsion if and only if for any $ a \in \kappa $, $ \pi(a\boldsymbol{x}) \in \pi(H_{J}) $. 
In particular whether a layer $ Z \in L(\mathcal{A}(K/\mathcal{O})) $ is $ \kappa $-torsion is independent of the choice of $ J $ such that $ Z \in T_{J}/\pi(H_{J}) $ by Lemma \ref{intersection}. 
Namely, $ Z $ is $ \kappa $-torsion if and only if 
\begin{align*}
\kappa \subseteq \Ann Z \coloneqq \Set{a \in \mathcal{O} | aZ \subseteq \pi(H_{Z})}, 
\end{align*}
where $ aZ \coloneqq \Set{\pi(a\boldsymbol{x}) \in (K/\mathcal{O})^{\ell} | \boldsymbol{x} \in Z} $.

\begin{definition}
For any $\kappa \in I(\mathcal{O})$, 
the \textbf{$\kappa$-torsion subposet} of $ L(\mathcal{A}(K/\mathcal{O})) $ is defined by 
\begin{align*}
L(\A (K / \mathcal{O}) )[\kappa] \coloneqq 
\Set{ Z \in L(\A (K / \mathcal{O}) ) | 
Z \text{ is } \kappa\text{-torsion } }. 
\end{align*}
\end{definition}

\begin{remark}
In the original definition \cite[Definition 4.4(1)]{tran2019combinatorics-joctsa}, which can be considered as the case $ \mathcal{O} = \mathbb{Z} $, the poset consists of layers containing a $ \kappa $-torsion element. 
At least when $ \kappa $ is a principal ideal, one can prove easily that a layer $ Z $ contains a $ \kappa $-torsion element if and only if $ Z $ itself is a $ \kappa $-torsion element (See also the formula (4.3) in the proof of Lemma 4.5 in \cite{tran2019combinatorics-joctsa} ). 
\end{remark}

\begin{proposition}\label{order ideal}
The $ \kappa $-torsion subposet $L(\A (K / \mathcal{O}) )[\kappa]$ is an order ideal of $L (\A (K / \mathcal{O}) )$. 
Namely if $ Z \leq W $ in $L (\A (K / \mathcal{O}) )$, or equivalently $ Z \supseteq W, $ and $ W \in L(\A (K / \mathcal{O}) )[\kappa] $, then $ Z \in L(\A (K / \mathcal{O}) )[\kappa] $. 
In particular, the restriction of the M\"{o}bius function on $ L(\A (K / \mathcal{O}) ) $ to $ L(\A (K / \mathcal{O}) )[\kappa] $ coincides with the M\"{o}bius function on $ L(\A (K / \mathcal{O}) )[\kappa] $. 
\end{proposition}
\begin{proof}
Let $ \pi(\boldsymbol{x}) \in W $. 
Then $ W = \pi(\boldsymbol{x}) + \pi(H_{W}) $ and $ Z = \pi(\boldsymbol{x}) + \pi(H_{Z}) $ since $ Z \supseteq W $. 
Let $ a \in \kappa $. 
Since $ W $ is $ \kappa $-torsion, $ \pi(a\boldsymbol{x}) \in \pi(H_{W}) \subseteq \pi(H_{Z}) $. 
Therefore $ Z $ is $ \kappa $-torsion. 
\end{proof}

Given an $ \mathcal{O} $-module $ M $ and $ \kappa \in I(\mathcal{O}) $, let $ M[\kappa] $ denote the submodule consisting of $ \kappa $-torsion elements. 

\begin{lemma}\label{ideal quotient gcd}
For any $\kappa$, $\mathfrak{a} \in I (\mathcal{O})$, 
$\left( \mathcal{O} / \mathfrak{a} \right) [\kappa] = \mathcal{O}/(\kappa + \mathfrak{a})$. 
\end{lemma}
\begin{proof}
Since $ \mathcal{O} $ is a Dedekind domain, $ \mathcal{O}/(\kappa + \mathfrak{a}) \simeq \mathfrak{a} (\kappa + \mathfrak{a})^{-1} / \mathfrak{a} $. 
Moreover, since $\left( \mathcal{O} / \mathfrak{a} \right) [\kappa] = 
\Set{ [\omega]_{\mathfrak{a}} \in \mathcal{O}/\mathfrak{a} | a \omega \in \mathfrak{a} \text{ for all } a \in \kappa }$, 
it suffices to show that 
\begin{align*}
(\mathfrak{a} \colon \kappa) \coloneqq \Set{ \omega \in \mathcal{O} | a \omega \in \mathfrak{a} \text{ for all } a \in \kappa } = \mathfrak{a} (\kappa + \mathfrak{a})^{-1}. 
\end{align*}

Let $\omega = \sum_i x_i y_i \in \mathfrak{a} (\kappa + \mathfrak{a})^{-1}$ ($x_i \in \mathfrak{a}, y_i \in (\kappa + \mathfrak{a})^{-1}$). 
Since $ (\kappa + \mathfrak{a})^{-1} \subseteq \mathfrak{a}^{-1} $, we have $ \omega \in \mathfrak{a}(\kappa + \mathfrak{a})^{-1} \subseteq \mathfrak{a}\mathfrak{a}^{-1} = \mathcal{O} $. 
Since $(\kappa + \mathfrak{a})^{-1} \subseteq \kappa^{-1}$, for any $a \in \kappa$, 
\begin{align*}
a \left( \sum_i x_i y_i \right) = \sum_i x_i (a y_i) \in \mathfrak{a} \mathcal{O} = \mathfrak{a}.
\end{align*}
Therefore $ \omega  \in (\mathfrak{a} \colon \kappa) $ and hence $(\mathfrak{a} \colon \kappa) \supseteq \mathfrak{a} (\kappa + \mathfrak{a})^{-1} $. 

Let $\omega \in (\mathfrak{a} \colon \kappa)$. Since $\mathcal{O} = (\kappa + \mathfrak{a}) (\kappa + \mathfrak{a})^{-1}$, there exist $x_i \in (\kappa + \mathfrak{a})$, 
$y_i \in (\kappa + \mathfrak{a})^{-1}$ such that $1 = \sum_i x_i y_i$. 
Since $\omega x_i \in \omega (\kappa + \mathfrak{a}) \subseteq \omega \kappa + \mathfrak{a} \subseteq \mathfrak{a}$, we have 
\begin{align*}
\omega = \omega \left( \sum_i x_i y_i \right) = \sum_i  ( \omega x_i) y_i \in \mathfrak{a} (\kappa + \mathfrak{a})^{-1}. 
\end{align*}
Hence $(\mathfrak{a} \colon \kappa) \subseteq \mathfrak{a} (\kappa + \mathfrak{a})^{-1}$. Therefore $(\mathfrak{a} \colon \kappa) = \mathfrak{a} (\kappa + \mathfrak{a})^{-1} $. 
\end{proof}

\begin{theorem} \label{num. of conn. comp.'s}
Let $\varnothing \ne J \subseteq [n]$. For any $\kappa \in I(\mathcal{O})$, 
\begin{align*}
\left( T_J / \pi (H_J) \right) [\kappa] \simeq 
\bigoplus_{i=1}^{r(J)} \mathcal{O} / \left( \kappa + \mathfrak{d}_{J, i} \right). 
\end{align*}
In particular, $\left| \left( T_J / \pi (H_J) \right) [\kappa] \right| = \prod_{i=1}^{r(J)} N \left( \kappa + \mathfrak{d}_{J, i} \right) = m(J,\kappa)$, where $ \mathfrak{d}_{J, 1}, \dots, \mathfrak{d}_{J, r(J)} $ are the ideals  determined in Lemma \ref{structure of coker phi_J}. 
\end{theorem}
\begin{proof}
By Lemma \ref{hissu} and \ref{ideal quotient gcd}, we obtain 
\begin{align*}
\left( T_J / \pi (H_J) \right) [\kappa] 
\simeq 
\bigoplus_{i=1}^{r(J)} \left( \mathcal{O} / \mathfrak{d}_{J, i} \right) [\kappa]
\simeq \bigoplus_{i=1}^{r(J)} \mathcal{O}/(\kappa + \mathfrak{d}_{J, i}). 
\end{align*}
\end{proof}

\begin{corollary}
For any $ \kappa \in I(\mathcal{O}) $, $ L(\mathcal{A}(K/\mathcal{O}))[\kappa] = L(\mathcal{A}(K/\mathcal{O}))[\kappa + \rho_{\mathcal{A}}] $, where $ \rho_{\mathcal{A}} = \lcm \Set{\mathfrak{d}_{J, r(J)} | \varnothing \ne J \subseteq [n]} $ denotes the LCM-period of $ \mathcal{A} $
\end{corollary}
\begin{proof}
Since $ \kappa + \rho_{\mathcal{A}} + \mathfrak{d}_{J,i} = \kappa + \mathfrak{d}_{J,i} $ for each $ i $, the assertion holds. 
\end{proof}

\begin{corollary}\label{1st and last constituent}
$ L(\mathcal{A}(K/\mathcal{O}))[\langle 1 \rangle] \simeq L(\mathcal{A}(K)) $ and $ L(\mathcal{A}(K/\mathcal{O}))[\rho_{\mathcal{A}}] = L(\mathcal{A}(K/\mathcal{O})) $. 
\end{corollary}
\begin{proof}
A layer $ Z \in L(\mathcal{A}(K/\mathcal{O})) $ is $ \langle 1 \rangle $-torsion if and only if $ Z = \pi(H_{Z}) $. 
Therefore $ L(\mathcal{A}(K/\mathcal{O}))[\langle 1 \rangle] \simeq L(\mathcal{A}(K)) $ by Proposition \ref{pi preserves order}. 

By Theorem \ref{num. of conn. comp.'s}, we have $ T_{J}/\pi(H_{J})[\rho_{\mathcal{A}}] = T_{J}/\pi(H_{J}) $. 
Therefore $ L(\mathcal{A}(K/\mathcal{O}))[\rho_{\mathcal{A}}] = L(\mathcal{A}(K/\mathcal{O})) $. 
\end{proof}

\subsection{The constituents and the characteristic polynomials}

The characteristic polynomial $ \chi_{\mathcal{A}(K)}(t) $ is defined by 
\begin{align*}
\chi_{\mathcal{A}(K)}(t) \coloneqq \sum_{X \in L(\mathcal{A}(K))}\mu(H_{\varnothing}, X)t^{\dim X}, 
\end{align*}
where $ \mu $ denotes the M\"{o}bius function on $L (\A (K) )$. 

The characteristic polynomial $ \chi_{\mathcal{A}(K/\mathcal{O})}(t) $ is defined by 
\begin{align*}
\chi_{\mathcal{A}(K/\mathcal{O})}(t) \coloneqq \sum_{Z \in L(\mathcal{A}(K/\mathcal{O}))} \mu(T_{\varnothing}, Z)t^{\dim Z}, 
\end{align*}
where $ \mu $ denotes the M\"{o}bius function on $L (\A (K / \mathcal{O}) )$ and $ \dim Z = \dim H_{Z} $. 

\begin{definition}
For each $ \kappa \mid \rho_{\mathcal{A}} $, the \textbf{$ \kappa $-characteristic polynomial} of $ \mathcal{A}(K/\mathcal{O}) $ is defined by 
\begin{align*}
\chi_{\A (K / \mathcal{O})}^{\kappa}
(t) \coloneqq \sum_{Z \in L (\A (K / \mathcal{O}) ) [\kappa]} \mu (T_{\varnothing}, Z) t^{\dim Z}, 
\end{align*}
where $ \mu $ denotes the M\"{o}bius function on $L (\A (K / \mathcal{O}) )$. 
\end{definition}

By Proposition \ref{order ideal}, the restriction of the M\"{o}bius function on $L (\A (K / \mathcal{O}) )$ to $L (\A (K / \mathcal{O}) )[\kappa]$ coincides with  the M\"{o}bius function on $L (\A (K / \mathcal{O}) )[\kappa]$. 
Also, it follows that $ \chi_{\mathcal{A}(K/\mathcal{O})}^{\langle 1 \rangle}(t) = \chi_{\mathcal{A}(K)}(t) $ and $ \chi_{\mathcal{A}(K/\mathcal{O})}^{\rho_{\mathcal{A}}}(t) = \chi_{\mathcal{A}(K/\mathcal{O})}(t) $ by Corollary \ref{1st and last constituent}. 




\begin{example}\label{example2-2}
Let $ \mathcal{O} = \mathbb{Z}[\sqrt{-1}], K = \mathbb{Q}(\sqrt{-1}) $, and
\begin{align*}
\mathcal{A} = \Set{\begin{pmatrix}
1 \\ 1
\end{pmatrix}, \begin{pmatrix}
1 \\ -1
\end{pmatrix}, \begin{pmatrix}
1 \\ \sqrt{-1}
\end{pmatrix}, \begin{pmatrix}
1 \\ -\sqrt{-1}
\end{pmatrix}}
\end{align*}
as in Example \ref{example2}. 
Figure \ref{Fig:ex Hasse2} shows the Hasse diagram of $ L(\mathcal{A}(K/\mathcal{O})) $. 
\begin{figure}[t] 
\centering
\begin{tikzpicture}
\node (empty) at (0,-1) {$T_{\varnothing}$};
\node (T1) at (-6,1) {$ \{x_{1}+x_{2} \in \mathcal{O}\} $}; 
\node (T2) at (-2,1) {$ \{x_{1}-x_{2} \in \mathcal{O}\} $}; 
\node (T3) at ( 2,1) {$ \{x_{1}+\sqrt{-1}x_{2} \in \mathcal{O}\} $}; 
\node (T4) at ( 6,1) {$ \{x_{1}-\sqrt{-1}x_{2} \in \mathcal{O}\} $}; 
\node (1) at (-7,5) {$ \left(\dfrac{1}{2},\dfrac{1}{2} \right) $}; 
\node (2) at (-4.5,5) {$ \left(\dfrac{\sqrt{-1}}{2},\dfrac{\sqrt{-1}}{2} \right) $}; 
\node (3) at (-2.1,5) {$ (0,0) $}; 
\node (4) at ( 1,5) {$ \left(\dfrac{1+\sqrt{-1}}{2},\dfrac{1+\sqrt{-1}}{2}\right) $}; 
\node (5) at ( 4.8,5) {$ \left(\dfrac{\sqrt{-1}}{2},\dfrac{1}{2} \right) $}; 
\node (6) at ( 7.5,5) {$ \left(\dfrac{1}{2},\dfrac{\sqrt{-1}}{2} \right) $}; 
\draw (empty) -- (T1);
\draw (empty) -- (T2);
\draw (empty) -- (T3);
\draw (empty) -- (T4);
\draw (1) -- (T1);
\draw (1) -- (T2);
\draw (2) -- (T1);
\draw (2) -- (T2);
\draw (3) -- (T1);
\draw (3) -- (T2);
\draw (3) -- (T3);
\draw (3) -- (T4);
\draw (4) -- (T1);
\draw (4) -- (T2);
\draw (4) -- (T3);
\draw (4) -- (T4);
\draw (5) -- (T3);
\draw (5) -- (T4);
\draw (6) -- (T3);
\draw (6) -- (T4);
\end{tikzpicture}
\caption{The Hasse diagram of $ L(\mathcal{A}(K/\mathcal{O})) $ in Example \ref{example2-2} } \label{Fig:ex Hasse2}
\end{figure}
Recall that the LCM-period of $ \mathcal{A} $ is $ \rho_{\mathcal{A}} = \langle 2 \rangle = \mathfrak{p}^{2} $, where $ \mathfrak{p} = \langle 1+\sqrt{-1} \rangle $. 
The subposets of $ L(\mathcal{A}(K\mathcal{O})) $ are as follows: 
\begin{align*}
L(\mathcal{A}(K/\mathcal{O}))[\langle 1 \rangle] &= [T_{\varnothing}, (0,0)] \\
L(\mathcal{A}(K/\mathcal{O}))[\mathfrak{p}] &= [T_{\varnothing}, (0,0)] \cup \left[T_{\varnothing}, \left(\dfrac{\sqrt{-1}}{2},\dfrac{\sqrt{-1}}{2} \right)\right] \\
L(\mathcal{A}(K/\mathcal{O}))[\mathfrak{p}^{2}] &= L(\mathcal{A}(K/\mathcal{O})). 
\end{align*}
Therefore we obtain 
\begin{align*}
\chi_{\mathcal{A}(K/\mathcal{O})}^{\langle 1 \rangle}(t) = t^{2}-4t+3, \quad
\chi_{\mathcal{A}(K/\mathcal{O})}^{\mathfrak{p}}(t) = t^{2}-4t+6, \quad
\chi_{\mathcal{A}(K/\mathcal{O})}^{\mathfrak{p}^{2}}(t) = t^{2}-4t+10. 
\end{align*}
In particular, by Example \ref{example2}, we have $f^{\kappa} _{\A} (t) = \chi_{\A (K / \mathcal{O})}^{\kappa} (t)$ for each ideal 
$\kappa \mid \rho_{\mathcal{A}} = \mathfrak{p}^{2} $. 
\end{example}

\begin{example}\label{example1-2}
Let $ \mathcal{O} = \mathbb{Z}[\sqrt{-5}], K = \mathbb{Q}(\sqrt{-5}) $, and 
\begin{align*}
\A = \{ c_1, c_2 \}, 
\ \text{ where }
c_{1} = \begin{pmatrix}
2  \\
1 - \sqrt{-5}
\end{pmatrix}, \
c_{2} = \begin{pmatrix}
1 + \sqrt{-5} \\
3
\end{pmatrix}
\end{align*}
as in Example \ref{example1}. 
Then $H \coloneqq H_{\{ 1 \}} = H_{\{2\}} = H_{\{1, 2\}} = \langle (-3, 1+\sqrt{-5}) \rangle_{K} \subseteq K^{2} $ and 
we have 
\begin{align*}
T_J / \pi (H_J) 
= \begin{cases}
\{T_{\varnothing}\} & \text{ if } J = \varnothing, \\
\left\{ \overline{(0,0)}, \ \overline{(1/2, 0)} \right\}  & \text{ if } J = \{ 1 \}, \\
\left\{ \overline{(0,0)}, \ \overline{(0, 1/3)}, \ \overline{(0, 2/3)} \right\}   & \text{ if } J = \{ 2 \}, \\
\left\{ \overline{(0,0)} \right\} & \text{ if } J = \{ 1, 2 \}, 
\end{cases}
\end{align*}
where $ \overline{(a,b)} $ denotes $ \pi(a,b) + \pi(H) $. 
Figure \ref{Fig:ex Hasse} shows the Hasse diagram of the poset of layers $ L(\mathcal{A}(K/\mathcal{O})) $. 

\begin{figure}[t] 
\centering
\begin{tikzpicture}
  \node (0) at (-3,0) {$\overline{(0,0)}$};
  \node (11) at (-1,0) {$\overline{(1/2,0)}$};
  \node (21) at (1,0) {$\overline{(0,1/3)}$};
  \node (22) at (3,0) {$\overline{(0,2/3)}$};
  \node (empty) at (0,-2) {$T_{\varnothing}$};
  \draw (empty) -- (0);
  \draw (empty) -- (11);
  \draw (empty) -- (21);
  \draw (empty) -- (22);
\end{tikzpicture}
\caption{The Hasse diagram of $ L(\mathcal{A}(K/\mathcal{O})) $ in Example \ref{example1-2} } \label{Fig:ex Hasse}
\end{figure}

Recall the LCM-period of $ \mathcal{A} $ is $ \rho_{\mathcal{A}} = \mathfrak{p}\mathfrak{q} = \langle 1+\sqrt{-5} \rangle $, where $ \mathfrak{p} = \langle 2, 1-\sqrt{-5} \rangle, \mathfrak{q} = \langle 3, 1+\sqrt{-5} \rangle $. 
The torsion subposets of $ L(\A (K / \mathcal{O})) $ are as follows: 
\begin{gather*}
L(\A (K / \mathcal{O}))[\langle 1 \rangle] = \{ T_{\varnothing} , \overline{(0,0)} \}, \quad 
L(\A (K / \mathcal{O}))[\mathfrak{p}] = \{ T_{\varnothing} , \overline{(0,0)}, \overline{(1/2,0)} \}, \\
L(\A (K / \mathcal{O}))[\mathfrak{q}] = \{ T_{\varnothing} , \overline{(0,0)}, \overline{(0,1/3)}, \overline{(0,2/3)} \}, \quad
L(\A (K / \mathcal{O})) [\mathfrak{p} \mathfrak{q}] 
= L(\A (K / \mathcal{O})).
\end{gather*}

Since $\dim T_{\varnothing} = 2$ and $\dim Z = \dim H = 1$ for each $Z \in L(\A (K / \mathcal{O}))  \setminus \{ T_{\varnothing} \}$, we obtain 
\begin{alignat*}{2}
\chi_{\A (K / \mathcal{O})}^{\langle 1 \rangle} (t) 
&= t^2 - t, 
& \qquad
\chi_{\A (K / \mathcal{O})}^{\mathfrak{p}} (t) 
&= t^2 - 2t, 
\\
\chi_{\A (K / \mathcal{O})}^{\mathfrak{q}} (t) 
&= t^2 - 3t, 
&
\chi_{\A (K / \mathcal{O})}^{\mathfrak{p}\mathfrak{q}} (t) 
&= t^2 - 4t.
\end{alignat*}
In particular, by Example \ref{example1}, we have $f^{\kappa} _{\A} (t) = \chi_{\A (K / \mathcal{O})}^{\kappa} (t)$ for each ideal 
$\kappa \mid \rho_{\mathcal{A}} = \mathfrak{p} \mathfrak{q} $. 
\end{example}

In the rest of this subsection, we will prove that for each ideal $\kappa \mid \rho_{\mathcal{A}}$, 
the $\kappa$-constituent $f^{\kappa}_{\mathcal{A}} (t)$ of the characteristic quasi-polynomial $\chi_\A^{\mathrm{quasi}} $ 
coincides with the $\kappa$-characteristic polynomial $\chi_{\A (K / \mathcal{O})}^{\kappa} (t)$ of $\A (K / \mathcal{O})$. 

\begin{definition}
Given a layer $ Z \in L(\mathcal{A}(K/\mathcal{O})) $, define $ D(Z) $ by 
\begin{align*}
D(Z) \coloneqq \Set{J \subseteq [n] | Z \in T_{J}/\pi(H_{J})}. 
\end{align*}
\end{definition}

\begin{lemma} \label{D(Z)}
Let $ Z \in L(\mathcal{A}(K/\mathcal{O})) $. 
Then the following hold. 
\begin{enumerate}[(i)]
\item\label{D(Z) 1} If $ J \in D(Z) $, then $ H_{J} = H_{Z} $. 
\item\label{D(Z) 2} $ J_{Z} = \Set{j \in [n] | T_{j} \supseteq Z} $ is a unique maximal element of $ D(Z) $. 
\item\label{D(Z) 3} $ \displaystyle\bigsqcup_{\substack{W \in L(\mathcal{A}(K/\mathcal{O})) \\ W \supseteq Z}}D(W) = \Set{J \subseteq [n] | J \subseteq J_{Z}}.  $
\end{enumerate}
\end{lemma}
\begin{proof}
(\ref{D(Z) 1}) It follows from Lemma \ref{intersection}. 

(\ref{D(Z) 2}) Let $ J \in D(Z) $. 
Then $ T_{J} \supseteq Z $ and hence $ T_{j} \supseteq Z $ for any $ j \in J $. 
Therefore $ J \subseteq J_{Z} $ and hence the maximality holds. 
Also, we have $ H_{Z} = H_{J} \supseteq H_{J_{Z}} $. 

For each $ j \in J_{Z} $, we have $ T_{j} \supseteq Z $ and hence $ T_{J_{Z}} \supseteq Z  $. 
By Proposition \ref{layer connectedness}, there exists a layer $ W \in T_{J_{Z}}/\pi(H_{J_{Z}}) $ such that $ W \supseteq Z $. 
Therefore we have $ H_{J_{Z}} \supseteq H_{Z} $ by Lemma \ref{interval}(\ref{interval 1}). 
Thus $ H_{J_{Z}} = H_{Z} $ and $ J_{Z} \in D(Z) $. 

(\ref{D(Z) 3}) First, we prove that the union is disjoint.  
Suppose that $ W_{1} \cap W_{2} \supseteq Z $ and $ W_{1} \neq W_{2} $. 
Assume that there exists $ J \in D(W_{1}) \cap D(W_{2}) $. 
Then $ H_{W_{1}} = H_{J} = H_{W_{2}} $. 
By Lemma \ref{interval}(\ref{interval 2}), we have $ W_{1} = W_{2} $, a contradiction. 
Thus $D(W_{1}) \cap D(W_{2}) = \varnothing$. 

Second, let $J \in D(W)$ such that $W \supseteq Z$. 
We will prove $ J \subseteq J_{Z} $. 
By (\ref{D(Z) 1}) and Lemma \ref{interval}(\ref{interval 1}), $H_{J_Z} = H_{Z} \subseteq H_{W} = H_J$. 
Thus $H_{J \cup J_Z} = H_{J} \cap H_{J_Z} = H_{J_Z}$. 
Furthermore, since $Z \subseteq T_{J_Z} $ and $Z \subseteq W \subseteq T_J$, we have 
$Z \subseteq T_{J_Z} \cap T_{J} 
= T_{J \cup J_Z}$. 
Therefore $ Z $ can be written as $ Z = \pi(\boldsymbol{x}) + \pi(H_{J \cup J_{Z}}) $ with $ \pi(\boldsymbol{x}) \in T_{J \cup J_{Z}} $ and hence $J \cup J_Z \in D(Z)$. 
By (\ref{D(Z) 2}), we have $J \cup J_Z \subseteq J_Z$. 
Thus $J \subseteq J_Z$. 

Finally, suppose that $J \subseteq J_Z$. 
Then we have $ Z \subseteq T_{J_{Z}} \subseteq T_{J} $ and $ H_{J_{Z}} \subseteq H_{J} $. 
Let $ \pi(\boldsymbol{x}) \in Z $ and $ W \coloneqq \pi(\boldsymbol{x}) + \pi(H_{J}) $. 
Then $Z = \pi (\boldsymbol{x}) + \pi (H_{J_Z}) \subseteq \pi (\boldsymbol{x}) + \pi (H_{J}) = W$ and $W = \pi (\boldsymbol{x}) + \pi (H_{J}) \in T_J / \pi (H_J)$, that is, $J \in D(W)$. 
\end{proof}

\begin{lemma}[See also {\cite[Lemma 5.5]{moci2012tutte-totams}}] \label{Moci lemma}
For any $Z \in L(\A (K / \mathcal{O}) )$, 
\begin{align*}
\mu (T_{\varnothing}, Z) = \sum_{J \in D(Z)} (-1)^{|J|}. 
\end{align*}
\end{lemma}

\begin{proof}
We prove it by induction on the codimension of $Z$. 
If $\codim Z = 0$ then $Z = T_{\varnothing}$ and $D(T_{\varnothing}) = \{ \varnothing \}$. 
Hence the statement holds. 
We assume that $\codim Z \geq 1$. 
By the definition of the M\"{o}bius function, the induction hypothesis, and Lemma \ref{D(Z)}(\ref{D(Z) 3}), we have 
\begin{align*}
\mu (T_{\varnothing}, Z) 
&= - \sum_{W \supsetneq Z} \mu (T_{\varnothing}, W) \\
&= - \sum_{W \supsetneq Z} \sum_{J \in D(W)} (-1)^{|J|} \\
&= \sum_{J \in D(Z)}(-1)^{|J|} - \sum_{W \supseteq Z} \sum_{J \in D(W)} (-1)^{|J|}  \\
&= \sum_{J \in D(Z)}(-1)^{|J|} - \sum_{J \subseteq J_{Z}}(-1)^{|J|} \\
&= \sum_{J \in D(Z)}(-1)^{|J|}. 
\end{align*}
\end{proof}

\begin{theorem}[Generalization of Theorem \ref{TY19}] \label{main thm2}
Let $\kappa \mid \rho_{\mathcal{A}}$. Then 
\begin{align*}
f^{\kappa}_{\mathcal{A}} (t) = \chi_{\A (K / \mathcal{O})}^{\kappa} (t). 
\end{align*}
\end{theorem}

\begin{proof}
By Theorem \ref{main thm}, 
the $\kappa$-constituent $f^{\kappa}_{\mathcal{A}} (t)$ of the characteristic quasi-polynomial $\chi_\A^{\mathrm{quasi}} $ is given by 
\begin{align*}
f_{\A}^{\kappa} (t) 
&= \sum_{J \subseteq [n]} (-1)^{|J|} \, m(J, \kappa) t^{\ell - r(J)} \\
&= t^{\ell} + 
\sum_{r = 0}^{\ell - 1}
\left(
\sum_{\substack{J \subseteq [n] \\  \ell - r(J) = r}} 
(-1)^{|J|} \, m(J, \kappa)
\right)
t^{r}. 
\end{align*}
The $\kappa$-characteristic polynomial  $\chi_{\A (K / \mathcal{O})}^{\kappa} (t)$ of $\A (K / \mathcal{O})$ is given by 
\begin{align*}
\chi_{\A (K / \mathcal{O})}^{\kappa} (t) = \sum_{Z \in L (\A (K / \mathcal{O})) [\kappa] } \mu (T_{\varnothing}, Z) t^{\dim Z} 
= t^{\ell} + 
\sum_{r = 0}^{\ell - 1}
\left(
\sum_{\substack{Z \in L (\A (K / \mathcal{O})) [\kappa] \\ \dim Z = r}} 
\mu (T_{\varnothing}, Z)
\right)
t^{r}. 
\end{align*}
Since if $J \in D(Z)$ then $\dim Z = \dim H_{J} = \ell - r(J)$, 
by Theorem \ref{num. of conn. comp.'s} and Lemma \ref{Moci lemma}, for each $0 \leq r \leq \ell - 1$, 
\begin{align*}
\sum_{\substack{Z \in L (\A (K / \mathcal{O})) [\kappa] \\ \dim Z = r}} 
\mu (T_{\varnothing}, Z)
&= \sum_{\substack{Z \in L (\A (K / \mathcal{O})) [\kappa] \\ \dim Z = r}} 
\ 
\sum_{J \in D(Z)} (-1)^{|J|} \\
&= \sum_{\substack{J \subseteq [n] \\ \ell - r(J) = r}} 
\ 
\sum_{\substack{Z \in L (\A (K / \mathcal{O})) [\kappa] \\ J \in D(Z)}} 
(-1)^{|J|} \\
&= \sum_{\substack{J \subseteq [n] \\ \ell - r(J) = r}}  (-1)^{|J|} \, m(J, \kappa).
\end{align*}
Therefore we obtain $f_{\A}^{\kappa} (t)  = \chi_{\A (K / \mathcal{O})}^{\kappa} (t) $. 
\end{proof}

\section{Minimality of the LCM-period} \label{proof of minimality}

For a layer $Z \in L(\A (K/\mathcal{O}))$, put 
\begin{align*}
\tau (Z) \coloneqq \Ann Z = \Set{ a \in \mathcal{O} | a Z \subseteq \pi (H_{Z}) }.  
\end{align*}

Note that $ Z \in L(\A (K/\mathcal{O}))[\kappa] $ if and only if $ \tau (Z) \mid \kappa $. 

\begin{lemma}[See also {\cite[Lemma 6.2]{higashitani2022period-imrn}}] \label{tau lemma 2}
Let $\varnothing \ne J \subseteq [n]$. 
\begin{enumerate}[(i)]
\item 
$\tau (Z) \mid \mathfrak{d}_{J, r(J)} $ for any $Z \in T_J / \pi (H_J)$. 
\item 
$\tau (Z) = \mathfrak{d}_{J, r(J)} $ for some $Z \in T_J / \pi (H_J)$. 
\end{enumerate}
\end{lemma}

\begin{proof}
It follows immediately from Lemma \ref{hissu}. 
\end{proof}

For $ r \in \{0, \dots, \ell-1 \} $ we define 
\begin{align*}
L^r (\A (K/\mathcal{O})) &\coloneqq \Set{ Z \in L(\A (K/\mathcal{O})) | \dim Z = r } 
\end{align*}
and 
\begin{align*}
d_r (\mathfrak{a}) \coloneqq \sum_{\substack{\varnothing \ne J \subseteq [n] \\ r(J) = \ell - r}} 
(-1)^{|J|} m(J, \mathfrak{a}). 
\end{align*}

By the proof of Theorem \ref{main thm2}, we have 
\begin{align*}
f_{\A}^{\kappa} (t) 
= t^{\ell} + \sum_{r=0}^{\ell-1} d_r (\kappa) t^r 
\quad \text{ and } \quad 
d_r (\kappa) = 
\sum_{\substack{Z \in L^r(\A (K/\mathcal{O})) \\ \tau(Z) \mid \kappa}} \mu (T_{\varnothing}, Z),
\end{align*}
for each $ \kappa \mid \rho_{\mathcal{A}} $.

\begin{lemma}[See also {\cite[Claim 6.3]{higashitani2022period-imrn}}] \label{minimality rho_r}
For each $ r \in \{0, \dots, \ell-1 \} $, 
$d_r (\mathfrak{a}) $ is a quasi-constant (quasi-monomial of degree $ 0 $) with the minimum period 
\begin{align*}
\rho_{r} \coloneqq \lcm \Set{ \tau (Z) | Z \in L^r (\A (K/\mathcal{O})) }. 
\end{align*}
\end{lemma}

\begin{proof}
We first prove that $\rho_{r}$ is a period of $d_r (\mathfrak{a}) $. 
By Lemma \ref{card of ker}, 
$m(J, \mathfrak{a})$ has the minimum period 
$\mathfrak{d}_{J, r(J)}$. 
Hence $d_r (\mathfrak{a}) $ has a period 
\begin{align*}
\lcm \Set{ \mathfrak{d}_{J, r(J)}  | \varnothing \ne J \subseteq [n], \ r(J) = \ell - r }. 
\end{align*}
By Lemma \ref{tau lemma 2}, this period equals $ \lcm \Set{ \tau(Z) | Z \in L^r (\A (K/\mathcal{O})) } = \rho_{r} $. 

Next we show the minimality of $\rho_{r}$. 
Let $\rho$ be the minimum period of $d_r (\mathfrak{a}) $. 
Then $\rho \mid \rho_{r}$ and $d_r (\rho) = d_r (\rho_{r})$. 
By the definition of $\rho_{r}$, we have $\tau (Z) \mid \rho_{r}$ for any $Z \in L^r (\A (K/\mathcal{O}))$. 
Hence we obtain 
\begin{align*}
\sum_{Z \in L^r (\A (K/\mathcal{O}))} (-1)^{\ell - r} \mu (T_{\varnothing}, Z)
&= \sum_{\substack{Z \in L^r(\A (K/\mathcal{O})) \\ \tau(Z) \mid \rho_{r}}} (-1)^{\ell - r} \mu (T_{\varnothing}, Z)  \\
&= (-1)^{\ell - r} d_r (\rho_{r}) 
= (-1)^{\ell - r} d_r (\rho) \\
&= \sum_{\substack{Z \in L^r (\A (K/\mathcal{O})) \\ \tau(Z) \mid \rho}} (-1)^{\ell - r} \mu (T_{\varnothing}, Z). 
\end{align*}
Since $(-1)^{\ell - r} \mu (T_{\varnothing}, Z) > 0$ by Lemma \ref{interval} (\ref{interval 3}), 
we have 
\begin{align*}
L^r (\A (K/\mathcal{O})) = \Set{Z \in L^r (\A (K/\mathcal{O})) \ | \ \tau(Z) \mid \rho}. 
\end{align*}
Therefore $\rho_{r} \mid \rho$ and hence $\rho_{r} = \rho$. 
\end{proof}

\begin{theorem}[Generalization of Theorem \ref{HTY21}] \label{main thm3}
The LCM-period $\rho_{\mathcal{A}}$ is the minimum period of the characteristic quasi-polynomial $\chi_{\A}^{\mathrm{quasi}}$. 
\end{theorem}

\begin{proof}
The characteristic quasi-polynomial 
$\chi_{\A}^{\mathrm{quasi}}$ is given by 
\begin{align*}
\chi_{\A}^{\mathrm{quasi}}(\mathfrak{a})
= N(\mathfrak{a})^{\ell} 
+ \sum_{r = 0}^{\ell - 1} d_r (\mathfrak{a}) N (\mathfrak{a})^{r}. 
\end{align*}
By Lemma \ref{minimality rho_r}, each coefficient $d_r (\mathfrak{a})$ has the minimum period $\rho_{r}$. 
Hence the characteristic quasi-polynomial 
$\chi_{\A}^{\mathrm{quasi}}$ 
has the minimum period 
$
\lcm \Set{ \rho_{r} | 0 \leq r \leq \ell - 1 }
$. 
By Lemma \ref{tau lemma 2}, we obtain 
\begin{align*}
\rho_\mathcal{A} 
&= \lcm \Set{ \mathfrak{d}_{J, r(J)}  | \varnothing \ne J \subseteq [n] }
= \lcm \Set{ \tau (Z) | Z \in L(\A (K/\mathcal{O})) \setminus \left\{ T_{\varnothing} \right\} } \\
&= \lcm \Set{ \rho_{r} | 0 \leq r \leq \ell - 1 }. 
\end{align*}
\end{proof}

\section{Localizations of the base rings}\label{localizations of the base rings}

Let $S$ be a multiplicative subset of $\mathcal{O}$ and $ \mathcal{O}_{S} $ denotes the localization of $ \mathcal{O} $ with respect to $ S $. 
Since $ \mathcal{O} $ is a Dedekind domain, so is $ \mathcal{O}_{S} $. 

Let $ I_{S}(\mathcal{O}) \coloneqq \Set{\mathfrak{A} \cap \mathcal{O} | \mathfrak{A} \in I(\mathcal{O}_{S})} $, where $ I(\mathcal{O}_{S}) $ denotes the set of nonzero ideals of $ \mathcal{O}_{S} $. 
By standard ring theory, we have a one-to-one correspondence between ideals in $ I_{S}(\mathcal{O}) $ and $ I(\mathcal{O}_{S}) $ as follows: 
\begin{align*}
\begin{array}{rcl}
I_{S}(\mathcal{O}) & \longleftrightarrow & I(\mathcal{O}) \\
\mathfrak{a} & \longmapsto & \mathfrak{a}\mathcal{O}_{S} \\
\mathfrak{A} \cap \mathcal{O} & \longmapsfrom & \mathfrak{A}. 
\end{array}
\end{align*}

Note that this correspondence preserves prime ideals and 
\begin{align*}
I_{S}(\mathcal{O}) \cap \Spec(\mathcal{O}) = \Set{\mathfrak{p} \in I(\mathcal{O}) \cap \Spec(\mathcal{O}) | \mathfrak{p} \cap S = \varnothing}, 
\end{align*}
where $ \Spec(\mathcal{O}) $ denotes the set of prime ideals of $ \mathcal{O} $. 

\begin{lemma} \label{Na = NaOS}
The following statements hold: 
\begin{enumerate}[(i)]
\item \label{Na = NaOS 1}
For any $s \in S$ and $\mathfrak{a} \in I_S (\mathcal{O})$, $ \mathfrak{a} + \langle s \rangle = \langle 1 \rangle $. 

\item \label{Na = NaOS 2}
For any ideal $\mathfrak{a} \in I_S (\mathcal{O})$, 
we have $\mathcal{O}_S / \mathfrak{a}\mathcal{O}_S \simeq \mathcal{O} / \mathfrak{a} $. 
In particular, $N (\mathfrak{a} \mathcal{O}_S) = N (\mathfrak{a})$. 
Moreover, $ \mathcal{O}_{S}/\mathfrak{A} \simeq \mathcal{O}/(\mathfrak{A} \cap \mathcal{O}) $ and $ N(\mathfrak{A}) = N(\mathfrak{A} \cap \mathcal{O}) $ for each $ \mathfrak{A} \in I(\mathcal{O}_{S}) $. 

\item \label{Na = NaOS 3}
$ \mathcal{O}_{S} $ is residually finite. 
\end{enumerate}
\end{lemma}

\begin{proof}
(\ref{Na = NaOS 1}) 
Since $ \mathcal{O}_{S} $ is a Dedekind domain, the ideal $ \mathfrak{a}\mathcal{O}_{S} $ is decomposed into a product of prime ideals as $ \mathfrak{a}\mathcal{O}_{S} = (\mathfrak{p}_{1}\mathcal{O}_{S}) \cdots (\mathfrak{p}_{r}\mathcal{O}_{S}) $, where $ \mathfrak{p}_{i} \in I_{S}(\mathcal{O}) \cap \Spec(\mathcal{O}) $. 
Then $ \mathfrak{a}\mathcal{O}_{S} = (\mathfrak{p}_{1}\mathcal{O}_{S}) \cdots (\mathfrak{p}_{r}\mathcal{O}_{S}) = (\mathfrak{p}_{1} \cdots \mathfrak{p}_{r})\mathcal{O}_{S} $ and hence $ \mathfrak{a} = \mathfrak{p}_{1} \cdots \mathfrak{p}_{r} $. 
Since $ \mathcal{O} $ is a Dedekind domain and $ \mathfrak{p}_{i} $ is a prime ideal, $ \mathfrak{p}_{i} + \langle s \rangle $ equals $ \mathfrak{p}_{i} $ or $ \langle 1 \rangle $. 
If $ \mathfrak{p}_{i} + \langle s \rangle = \mathfrak{p}_{i} $, then $ s \in \mathfrak{p}_{i} $, which contradicts to $ \mathfrak{p}_{i} \cap S = \varnothing $. 
Therefore $ \mathfrak{p}_{i} + \langle s \rangle = \langle 1 \rangle $. 
Thus $ \mathfrak{a} + \langle s \rangle = \langle 1 \rangle $. 

(\ref{Na = NaOS 2})
By (\ref{Na = NaOS 1}), every element of $ S $ is a unit of $ \mathcal{O}/\mathfrak{a} $. 
Therefore $ \mathcal{O}_{S}/\mathfrak{a}\mathcal{O}_{S} \simeq \left(\mathcal{O}/\mathfrak{a}\right) \otimes_{\mathcal{O}} \mathcal{O}_{S} \simeq \mathcal{O}/\mathfrak{a} $. 

(\ref{Na = NaOS 3})
This follows immediately from (\ref{Na = NaOS 2}). 
\end{proof}

Let $ \mathcal{A}_{S} $ denote the finite set $ \mathcal{A} $ considered as a subset of $ \mathcal{O}_{S} $ via the natural inclusion $ \mathcal{O} \subseteq \mathcal{O}_{S} $. 
Since $ \mathcal{O}_{S} $ is a residually finite Dedekind domain by Lemma \ref{Na = NaOS}(\ref{Na = NaOS 3}), we can consider the characteristic quasi-polynomial and the torsion arrangement of $ \mathcal{A}_{S} $. 

\subsection{The characteristic quasi-polynomial of $\A_{S}$}

The characteristic quasi-polynomial $ \chi_{\mathcal{A}_{S}}^{\mathrm{quasi}} $ of $ \mathcal{A}_{S} $ is a quasi-polynomial on $ I(\mathcal{O}_{S}) $ defined by $\chi^{\mathrm{quasi}}_{\A_{S}} (\mathfrak{A}) = \left| M (\A_{S} (\mathcal{O}_S / \mathfrak{A})) \right|$. 
The following proposition shows that $ \chi_{\mathcal{A}_{S}}^{\mathrm{quasi}} $ is equivalent to the restriction of $ \chi_{\mathcal{A}}^{\mathrm{quasi}} $ to $ I_{S}(\mathcal{O}) $ via the correspondence described above.

\begin{theorem} \label{localization}
The following statements hold: 
\begin{enumerate}[(i)]
\item \label{localization 1} The LCM-period of $ \mathcal{A}_{S} $ is $ \rho_{\mathcal{A}_{S}} = \rho_{\mathcal{A}}\mathcal{O}_{S} $. 
\item \label{localization 2} For any $\mathfrak{A} \in I (\mathcal{O}_S)$, $\chi^{\mathrm{quasi}}_{\A_{S}} (\mathfrak{A}) = \chi^{\mathrm{quasi}}_{\A} (\mathfrak{A} \cap \mathcal{O})$. 
\item \label{localization 3} For any $\delta \mid \rho_{\mathcal{A}_{S}}$, $ f_{\mathcal{A}_{S}}^{\delta}(t) = f_{\mathcal{A}}^{\delta \cap \mathcal{O}}(t) $. 
\end{enumerate}
\end{theorem}

\begin{proof}
(\ref{localization 1})
For each nonempty subset $J \subseteq [n]$, we define $\phi_{J, S} \colon \mathcal{O}_S^{\ell} \to \mathcal{O}_S^{|J|}$ by 
$\boldsymbol{x} \mapsto \boldsymbol{x}  C_J $. 
Then we have $\coker \phi_{J, S} = \left( \coker \phi_J \right) \otimes_{\mathcal{O}} \mathcal{O}_S$. 
Hence by  Lemma \ref{structure of coker phi_J}, we obtain 
\begin{align*}
\coker \phi_{J, S} 
&= \left( \coker \phi_J \right) \otimes_{\mathcal{O}} \mathcal{O}_S 
\simeq 
\left( \bigoplus_{i = 1}^{r(J)} \mathcal{O} / \mathfrak{d}_{J, i} \oplus \mathfrak{d} \oplus \mathcal{O}^{|J| - r(J) -1} \right) 
\otimes_{\mathcal{O}} \mathcal{O}_S \\
&\simeq 
\bigoplus_{i = 1}^{r(J)} \mathcal{O}_S / \mathfrak{d}_{J, i} \mathcal{O}_S \oplus \mathfrak{d}\mathcal{O}_{S} \oplus \mathcal{O}_S^{|J| - r(J)-1}. 
\end{align*}
Therefore $ \rho_{\mathcal{A}_{S}} = \lcm \Set{\mathfrak{d}_{J, r(J)} \mathcal{O}_S | \varnothing \ne J \subseteq [n]} 
= \rho_{\mathcal{A}} \mathcal{O}_S $. 

(\ref{localization 2})
By Lemma \ref{Na = NaOS}(\ref{Na = NaOS 2}), $ \chi_{\mathcal{A}_{S}}^{\mathrm{quasi}}(\mathfrak{A}) = \left| M(\mathcal{A}_{S}(\mathcal{O}_{S}/\mathfrak{A})) \right| = \left| M(\mathcal{A}(\mathcal{O}/(\mathfrak{A} \cap \mathcal{O}))) \right| = \chi_{\mathcal{A}}^{\mathrm{quasi}}(\mathfrak{A} \cap \mathcal{O}) $. 

(\ref{localization 3})
Let $ \mathfrak{A} \in I(\mathcal{O}_{S}) $ be any ideal satisfying $ \mathfrak{A} + \rho_{\mathcal{A}_{S}} = \delta $. 
Then $ (\mathfrak{A} \cap \mathcal{O}) + \rho_{\mathcal{A}} = \delta \cap \mathcal{O} $. 
Therefore 
\begin{align*}
f_{\mathcal{A}_{S}}^{\delta}(N(\mathfrak{A})) 
= \chi_{\mathcal{A}_{S}}^{\mathrm{quasi}}(\mathfrak{A})
= \chi_{\mathcal{A}}^{\mathrm{quasi}}(\mathfrak{A} \cap \mathcal{O})
= f_{\mathcal{A}}^{\delta \cap \mathcal{O}}(N(\mathfrak{A\cap \mathcal{O}}))
= f_{\mathcal{A}}^{\delta \cap \mathcal{O}}(N(\mathfrak{A})). 
\end{align*}
Hence we obtain $ f_{\mathcal{A}_{S}}^{\delta}(t) = f_{\mathcal{A}}^{\delta \cap \mathcal{O}}(t) $. 
\end{proof}

\subsection{The torsion arrangement $\A_{S} (K/ \mathcal{O}_S)$}

The torsion arrangement $\A_{S} (K / \mathcal{O}_S)$ consists of 
\begin{align*}
T_{j, S} \coloneqq \Set{ \pi_S (\boldsymbol{x}) \in (K/\mathcal{O}_S)^{\ell} | \boldsymbol{x} c_j \equiv 0 \pmod{\mathcal{O}_S} }
\qquad (j \in [n]), 
\end{align*} 
where $\pi_S \colon K^{\ell} \to \left( K / \mathcal{O}_S \right)^{\ell}$ is the natural projection. 
For each $\varnothing \ne J \subseteq [n]$, let $T_{J, S} \coloneqq \bigcap_{j \in J} T_{j, S}$ and let 
$T_{\varnothing, S} \coloneqq (K/\mathcal{O}_S)^{\ell} $. 

The posets of layers of $ \mathcal{A}(K/\mathcal{O}) $ and $ \mathcal{A}_{S}(K/\mathcal{O}_{S}) $ are as follows: 
\begin{align*}
L(\mathcal{A}(K/\mathcal{O})) &= \Set{\pi(\boldsymbol{x}) + \pi(H_{J}) | J \subseteq [n], \pi(\boldsymbol{x}) \in T_{J}}, \\
L(\mathcal{A}_{S}(K/\mathcal{O}_{S})) &= \Set{\pi_{S}(\boldsymbol{x}) + \pi_{S}(H_{J}) | J \subseteq [n], \pi_{S}(\boldsymbol{x}) \in T_{J, S}}. 
\end{align*}

Let $ \eta_{S} \colon (K/\mathcal{O})^{\ell} \to (K/\mathcal{O}_{S})^{\ell} $ be the canonical $ \mathcal{O} $-homomorphism.
Since $\eta_{S} \circ \pi = \pi_S$, the $\mathcal{O}$-homomorphism $\eta_{S}$ induces the following order-preserving map: 
\begin{align*}
\begin{array}{rcl}
\eta_{S} \colon  L(\A (K / \mathcal{O})) & \longrightarrow & L(\A (K / \mathcal{O}_S)) \\
Z & \longmapsto & \eta_{S}(Z) = \Set{\eta_{S}(z) | z \in Z}. 
\end{array}
\end{align*}
Note that $\dim Z = \dim \eta_{S} (Z)$ since if $ Z = \pi(\boldsymbol{x}) + \pi(H_{J}) $ then $ \eta_{S}(Z) = \pi_{S}(\boldsymbol{x}) + \pi_{S}(H_{J}) $. 

The map $ \eta_{S} $ induces isomorphisms between torsion subposets as follows: 
\begin{theorem} \label{localization characteristic polynomial}
For any ideal $\delta \mid \rho_{\mathcal{A}_{S}}$, 
the restriction 
\begin{align*}
\eta_{S} \colon  L(\A (K / \mathcal{O}))[\delta \cap \mathcal{O}] \longrightarrow 
L(\A_{S} (K / \mathcal{O}_S))[\delta] 
\end{align*}
is an isomorphism as finite posets. 
In particular, 
\begin{align*}
L(\mathcal{A}_{S}(K/\mathcal{O}_{S})) \simeq L(\mathcal{A}(K/\mathcal{O}))[\rho_{\mathcal{A}}\mathcal{O}_{S} \cap \mathcal{O}]. 
\end{align*}
\end{theorem}
\begin{proof}
First, we show that if $ Z \in L(\A (K / \mathcal{O}))[\delta \cap \mathcal{O}] $, then $ \eta_{S}(Z) \in L(\A_{S} (K / \mathcal{O}_S))[\delta] $. 
Suppose that $ Z = \pi(\boldsymbol{x}) + \pi(H_{Z}) $. 
Then $ \eta_{S}(Z) = \pi_{S}(\boldsymbol{x}) + \pi_{S}(H_{Z}) $. 

Assume $ a/s \in \delta \ (a \in \mathcal{O}, s \in S) $. 
Since $ a \in \delta \cap \mathcal{O} $ and $ Z \in L(\A (K / \mathcal{O}))[\delta \cap \mathcal{O}] $, we have $ \pi(a\boldsymbol{x}) \in \pi(H_{Z}) $ and hence $ \pi_{S}(a\boldsymbol{x}) \in \pi_{S}(H_{Z}) $. 
Then there exists $ \boldsymbol{v} \in H_{Z} $ such that $ a\boldsymbol{x} - \boldsymbol{v} \in \mathcal{O}_{S}^{\ell} $. 
Therefore $ (a/s)\boldsymbol{x} - (1/s)\boldsymbol{v} \in \mathcal{O}_{S}^{\ell} $. 
Since $ (1/s)\boldsymbol{v} \in H_{Z} $, we have $ \pi_{S}((a/s)\boldsymbol{x}) \in \pi_{S}(H_{Z}) $. 
Thus $ \eta_{S}(Z) \in L(\A_{S} (K / \mathcal{O}_S))[\delta] $.

Secondly, we prove the injectivity. 
Suppose $ Z_{1}, Z_{2} \in L(\A (K / \mathcal{O}))[\delta \cap \mathcal{O}] $ and $ Z_{i} = \pi(\boldsymbol{x}_{i}) + \pi(H_{Z_{i}}) $ for $ i \in \{1,2\} $ and assume that $ \eta_{S}(Z_{1}) = \eta_{S}(Z_{2}) $. 
Then $ \pi_{S}(\boldsymbol{x}_{1}) + \pi_{S}(H_{Z_{1}}) = \pi_{S}(\boldsymbol{x}_{2}) + \pi_{S}(H_{Z_{2}}) $ and hence $ H \coloneqq H_{Z_{1}} = H_{Z_{2}} $ by Lemma \ref{intersection}. 
Therefore there exists $ \boldsymbol{v} \in H $ such that $ \boldsymbol{x}_{1} - \boldsymbol{x}_{2} - \boldsymbol{v} \in \mathcal{O}_{S}^{\ell} $. 
Hence there exists $ s \in S $ such that $ s\boldsymbol{x}_{1} - s\boldsymbol{x}_{2} - s\boldsymbol{v} \in \mathcal{O}^{\ell} $. 
Then we have $ \pi(s\boldsymbol{x}_{1}) \equiv \pi(s\boldsymbol{x}_{2}) \pmod{\pi(H)} $.

By Lemma \ref{Na = NaOS}(\ref{Na = NaOS 1}), we have $ (\delta \cap \mathcal{O}) + \langle s \rangle = \langle 1 \rangle $ and hence there exist $ d \in \delta \cap \mathcal{O} $ and $ a \in \mathcal{O} $ such that $ d + as = 1 $. 
Since $ \pi(d\boldsymbol{x}_{1}), \pi(d\boldsymbol{x}_{2}) \in \pi(H) $ by the assumption $ Z_{1}, Z_{2} \in L(\A (K / \mathcal{O}))[\delta \cap \mathcal{O}] $, we have  
\begin{align*}
\pi(\boldsymbol{x}_{1})
&= \pi(d\boldsymbol{x}_{1}) + \pi(as\boldsymbol{x}_{1}) 
\equiv a\pi(s\boldsymbol{x}_{1})  \\
&\equiv a\pi(s\boldsymbol{x}_{2}) 
= \pi(d\boldsymbol{x}_{2}) + \pi(as\boldsymbol{x}_{2}) 
= \pi(\boldsymbol{x}_{2}) \pmod{\pi(H)}. 
\end{align*} 
Thus $ Z_{1} = \pi(\boldsymbol{x}_{1}) + \pi(H) = \pi(\boldsymbol{x}_{2}) + \pi(H) = Z_{2} $ and hence the map $ \eta_{S} $ is injective. 

Finally, we prove the surjectivity. 
A layer in $ L(\mathcal{A}_{S}(K/\mathcal{O}_{S}))[\delta] $ belongs to $ (T_{J,S}/\pi_{S}(H_{J}))[\delta]  $ for some $ J \subseteq [n] $. 
From the discussion so far, when we restrict the map $ \eta_{S} $ to $ (T_{J}/\pi(H_{J}))[\delta \cap \mathcal{O}] $, we have an injection $ \eta_{J} \colon (T_{J}/\pi(H_{J}))[\delta \cap \mathcal{O}] \to (T_{J,S}/\pi_{S}(H_{J}))[\delta] $. 
By Theorem \ref{num. of conn. comp.'s} and Lemma \ref{Na = NaOS}(\ref{Na = NaOS 2}), 
\begin{align*}
|(T_{J}/\pi(H_{J}))[\delta \cap \mathcal{O}]| 
&= \prod_{i=1}^{r(J)}N((\delta \cap \mathcal{O}) + \mathfrak{d}_{J,i})
= \prod_{i=1}^{r(J)}N(\delta + \mathfrak{d}_{J,i}\mathcal{O}_{S}) 
= |(T_{J,S}/\pi_{S}(H_{J}))[\delta]|. 
\end{align*}
Therefore $ \eta_{J} $ is a bijection and hence $ \eta_{S} $ is surjective. 
\end{proof}

\begin{corollary}\label{localization torsion characteristic polynomials}
Let $ \delta \mid \rho_{\mathcal{A}_{S}} $. 
Then 
\begin{align*}
\chi_{\mathcal{A}_{S}(K/\mathcal{O}_{S})}^{\delta}(t) = \chi_{\mathcal{A}(K/\mathcal{O})}^{\delta \cap \mathcal{O}}(t). 
\end{align*}
\end{corollary}
\begin{proof}
It follows immediately from Theorem \ref{localization characteristic polynomial}. 
\end{proof}

\begin{remark}
Corollary \ref{localization torsion characteristic polynomials} also can be obtained by Theorem \ref{main thm2} and Theorem \ref{localization}(\ref{localization 3}). 
\end{remark}

\section{Computing methods for the characteristic quasi-polynomials}\label{computing}

\subsection{The LCM-periods}

Recall that the LCM-period $ \rho_{\mathcal{A}} $ is defined by 
\begin{align*}
\rho_{\mathcal{A}} = \lcm\Set{\mathfrak{d}_{J, r(J)} | \varnothing \neq J \subseteq [n]}, 
\end{align*}
where the ideals $ \mathfrak{d}_{J,i} $ are described in Lemma \ref{structure of coker phi_J}. 
In this subsection, we will prove that the size of $J$ can be restricted as follows: 
\begin{proposition}[See also {\cite[p.323 Formula (11)]{kamiya2008periodicity-joac}}] \label{restriction of LCM}
\begin{align*} 
\rho_{\mathcal{A}} = \operatorname{lcm} \Set{\mathfrak{d}_{J, r(J)} | J \subseteq [n], \, 1 \leq | J | \leq \operatorname{min} \{ \ell, n \}}. 
\end{align*}
\end{proposition}

In order to prove Proposition \ref{restriction of LCM}, we show the following lemma. 
\begin{lemma}[See also {\cite[Lemma 2.3]{kamiya2008periodicity-joac}}] \label{restriction}
Let $J_1$ and $J_2$ be two subsets of $[n]$ such that $J_1 \supseteq J_2$ and $r(J_1) = r(J_2)$. 
Then $\mathfrak{d}_{J_1, r(J_1)} \mid \mathfrak{d}_{J_2, r(J_2)}$. 
\end{lemma}
\begin{proof}
Since $ J_{1} \supseteq J_{2} $, we have $ T_{J_{1}} \subseteq T_{J_{2}} $ and $ H_{J_{1}} \subseteq H_{J_{2}} $. 
Moreover, since $ r(J_{1}) = r(J_{2}) $, we have $ H \coloneqq H_{J_{1}} = H_{J_{2}} $. 
Then we obtain an inclusion $ T_{J_{1}}/\pi(H) \subseteq T_{J_{2}}/\pi(H) $. 
Since $ \mathfrak{d}_{J_{i}, r(J_{i})} = \Ann(T_{J_{i}}/\pi(H)) $ for each $ i \in \{1,2\} $ by Lemma \ref{hissu}, we have $ \mathfrak{d}_{J_{2}, r(J_{2})} \subseteq \mathfrak{d}_{J_{1}, r(J_{1})} $. 
Thus $ \mathfrak{d}_{J_{1}, r(J_{1})} \mid \mathfrak{d}_{J_{2}, r(J_{2})} $. 
\end{proof}

\begin{proof}[Proof of Proposition \ref{restriction of LCM}]
Assume that $\ell < n$. Let $J$ be a subset of $[n]$ with $\ell < |J| \leq n$. 
Then we can choose a subset $J' \subseteq J$ such that $r(J') = |J'| = r(J) \leq \ell$. 
Lemma \ref{restriction} implies that $\mathfrak{d}_{J, r(J)} \mid \mathfrak{d}_{J', r(J')}$. 
This leads to the desired result. 
\end{proof}

\subsection{Coefficients}

Recall that the $ \kappa $-constituent of $ \chi_{\mathcal{A}}^{\mathrm{quasi}} $ is given by 
\begin{align*}
f_{\A}^{\kappa} (t) = \sum_{J \subseteq [n]} (-1)^{|J|} \, m(J, \kappa) t^{\ell - r(J)}
\end{align*}
by Theorem \ref{main thm}, where $m(J, \kappa) = \prod_{i = 1}^{r(J)} N \left( \kappa + \mathfrak{d}_{J, i} \right)$. 

\begin{proposition}[See also {\cite[Corollary 2.3]{kamiya2011periodicity-aoc}}] \label{cor1}
Suppose that $\kappa, \kappa' \in I(\mathcal{O})$ both divide $\rho_{\mathcal{A}}$ and assume that there exists $s \in \mathbb{Z}_{>0}$ such that 
$ \kappa + \mathfrak{d}_{J, r(J)} = \kappa' + \mathfrak{d}_{J, r(J)} $ for all $J \subseteq [n]$ with $|J| \leq s$. 
Then 
\begin{align*}
\deg \left(
f^\kappa_\A (t) - f^{\kappa'}_\A (t) 
\right) < \ell - s . 
\end{align*}
In particular, we have 
$\deg \left( 
f^\kappa_\A (t) - f^{\langle 1 \rangle}_\A (t) 
\right) < \ell - s  
$ if $\mathfrak{d}_{J, r(J)} + \kappa = \langle 1 \rangle $ for all $J \subseteq [n]$ with $|J| \leq s$. 
\end{proposition}

\begin{proof}
It is sufficient to show that 
$m (J, \kappa) = m(J, \kappa')$ for each $J \subseteq [n]$ with $r (J) \leq s$. 
Let $J'$ be a subset of $J$ such that $r(J') = |J'| = r (J) \leq s$. 
By the assumption, we have $\kappa + \mathfrak{d}_{J', r(J')} = \kappa' + \mathfrak{d}_{J', r(J')} $. 
By Lemma \ref{restriction}, $\mathfrak{d}_{J, r(J)} \mid \mathfrak{d}_{J', r(J')}$. 
Hence we obtain 
\begin{align*}
m (J, \kappa) 
&= \prod_{i=1}^{r(J)} N (\kappa + \mathfrak{d}_{J,i}) 
= \prod_{i=1}^{r(J)} N (\kappa + \mathfrak{d}_{J, r(J)} + \mathfrak{d}_{J,i}) 
= \prod_{i=1}^{r(J)} N (\kappa' + \mathfrak{d}_{J, r(J)} + \mathfrak{d}_{J,i}) \\ 
&= \prod_{i=1}^{r(J)} N (\kappa' + \mathfrak{d}_{J,i}) 
= m (J, \kappa') .
\end{align*}
\end{proof}

\begin{proposition}[See also {\cite[Corollary 2.4]{kamiya2011periodicity-aoc}}] \label{cor2}
Suppose that $\kappa, \kappa' \in I(\mathcal{O}) $ both divide $\rho_{\mathcal{A}}$ and that $\kappa + \kappa' = \langle 1 \rangle$. 
Assume that there exists $s \in \mathbb{Z}_{>0}$ such that 
$\kappa + \mathfrak{d}_{J, r(J)} = \langle 1 \rangle$ or $ \kappa' + \mathfrak{d}_{J, r(J)} = \langle 1 \rangle$ for all $J \subseteq [n]$ with $|J| \leq s$. 
Then 
\begin{align*}
\deg \left( 
f^{\langle 1 \rangle}_\A (t) + 
f^{\kappa \kappa'}_\A (t) 
- f^{\kappa}_\A (t) 
- f^{\kappa'}_\A (t) 
\right) < \ell - s . 
\end{align*}
\end{proposition}

\begin{proof}
For each $J \subseteq [n]$ with $r (J) \leq s$, it is sufficient to show that 
\begin{align*}
1 + m(J, \kappa \kappa') - m(J, \kappa) - m(J, \kappa') = 0. 
\end{align*}
Let $J'$ be a subset of $J$ such that $r(J') = |J'| = r (J) \leq s$. 
By the assumption, we have $\kappa + \mathfrak{d}_{J', r(J')} = \langle 1 \rangle $ or $ \kappa' + \mathfrak{d}_{J', r(J')} = \langle 1 \rangle$. 
By Lemma \ref{restriction}, $\mathfrak{d}_{J, r(J)} \mid \mathfrak{d}_{J', r(J')}$. 
Hence we have 
$\kappa + \mathfrak{d}_{J, r(J)} = \langle 1 \rangle$ or $ \kappa' + \mathfrak{d}_{J, r(J)} = \langle 1 \rangle$. 
Since $\mathfrak{d}_{J,i} \mid \mathfrak{d}_{J, r(J)}$ for each $i$, we have 
\begin{align*}
m(J, \kappa) = \prod_{i=1}^{r(J)} N (\kappa + \mathfrak{d}_{J, i}) = 1 \qquad \text{or} \qquad 
m(J, \kappa') = \prod_{i=1}^{r(J)} N (\kappa' + \mathfrak{d}_{J, i}) = 1. 
\end{align*}
Thus we have 
\begin{align*}
0 
&= \left( 1 - m(J, \kappa) \right) \left( 1 - m(J, \kappa') \right) 
= 1 - m(J, \kappa) - m(J, \kappa') + m(J, \kappa) m(J, \kappa') \\ 
&= 1 - m(J, \kappa) - m(J, \kappa') + \prod_{i=1}^{r(J)} N \left( \left( \kappa + \mathfrak{d}_{J, i} \right) \left( \kappa + \mathfrak{d}_{J, i} \right) \right). 
\end{align*}
Since $\kappa + \kappa' = \langle 1 \rangle$, 
\begin{align*}
\left( \kappa + \mathfrak{d}_{J, i} \right) \left( \kappa' + \mathfrak{d}_{J, i} \right) 
&= \kappa\kappa' + \mathfrak{d}_{J, i} \left( \kappa + \kappa' + \mathfrak{d}_{J, i} \right) 
= \kappa\kappa' + \mathfrak{d}_{J, i} \left( \langle 1 \rangle + \mathfrak{d}_{J, i} \right) 
= \kappa\kappa' + \mathfrak{d}_{J, i}. 
\end{align*}
Therefore 
\begin{align*}
\prod_{i=1}^{r(J)} N \left( \left( \kappa + \mathfrak{d}_{J, i} \right) \left( \kappa + \mathfrak{d}_{J, i} \right) \right)
= \prod_{i=1}^{r(J)}(\kappa\kappa' + \mathfrak{d}_{J,i})
= m(J, \kappa\kappa'), 
\end{align*}
which leads to the desired result. 
\end{proof}

\begin{proposition}[See also {\cite[Corollary 2.5]{kamiya2011periodicity-aoc}}] \label{cor3}
Suppose that $\kappa, \kappa' \in I(\mathcal{O})$ both divide $\rho_{\mathcal{A}}$ and that $\kappa + \kappa' = \langle 1 \rangle$. 
If $\mathfrak{d}_{J, r(J)}$ are powers of prime ideals for all $J$, we have 
$ 
f^{\kappa \kappa'}_\A (t) 
= f^{\kappa}_\A (t) 
+ f^{\kappa'}_\A (t) 
- f^{\langle 1 \rangle}_\A (t) $. 
\end{proposition}

\begin{proof}
Since $ \kappa + \kappa^{\prime} = \langle 1 \rangle $, it is always true that $ \kappa + \mathfrak{p}^{m} = \langle 1 \rangle $ or $ \kappa^{\prime} + \mathfrak{p}^{m} = \langle 1 \rangle $ for any prime ideal $ \mathfrak{p} $ and $ m \in \mathbb{Z}_{\geq 0} $. 
Therefore the claim follows immediately from Proposition \ref{cor2}. 
\end{proof}

\section{The characteristic quasi-polynomials of non-crystallographic root systems of types $\mathrm{H}_2$, $\mathrm{H}_3$, and $\mathrm{H}_4$} \label{Examples}

Let $ \Phi $ be an irreducible crystallographic root system and $ \Phi^{+} $ a positive system of $ \Phi $. 
Every positive root is expressed as a linear combination of the simple roots with integral coefficients. 
Gathering the coefficient column vectors, we obtain the set $ \mathcal{A}_{\Phi} $ consisting of integral column vectors. 
Kamiya, Takemura, and Terao \cite{kamiya2007characteristic-a, kamiya2010characteristic-alsas} computed the characteristic quasi-polynomial of $ \mathcal{A}_{\Phi} $ and its LCM-period explicitly by using the classification of root systems. 
Note that Suter \cite{suter1998number-moc} gave essentially the same calculation in terms of the number of lattice points in the fundamental alcoves (the Ehrhart quasi-polynomials). 

Now, by the results in this paper, we can consider the non-crystallographic cases. 
The irreducible non-crystallographic root systems are classified into one infinite family $ \mathrm{I}_{2}(m) $ together with two exceptional types $ \mathrm{H}_{3} $ and $ \mathrm{H}_{4} $. 
We consider the root systems of types $ \mathrm{H}_{2} \coloneqq I_{2}(5), \mathrm{H}_{3} $, and $ \mathrm{H}_{4} $. 

We assume that every root in $ \Phi_{\mathrm{H}_{n}} $ is normalized. 
Then every positive root is represented as a linear combination of the simple roots over $ \mathcal{O} \coloneqq \mathbb{Z}[\tau] $, where $ \tau $ denotes the golden ratio $ \tau \coloneqq \frac{1+\sqrt{5}}{2} $. 
The ring $ \mathcal{O} $ is a residually finite Dedekind domain since it is the ring of integers of the quadratic field $ K \coloneqq \mathbb{Q}(\sqrt{5}) $. 

We calculate the characteristic quasi-polynomials and the LCM-periods (the minimum periods) of the root systems of types $ \mathrm{H}_{2}, \mathrm{H}_{3} $, and $ \mathrm{H}_{4} $, using with SageMath \cite{sagemath} and propositions obtained in Section \ref{computing}. 

\subsection{The characteristic quasi-polynomial $\chi_{\mathrm{H}_{2}}^{\mathrm{quasi}}$}
Suppose that 
\begin{align*}
\Phi_{\mathrm{H}_{2}} 
= \Set{ \pm\left(\cos\dfrac{2\pi k}{5}, \sin\dfrac{2\pi k}{5}\right) | 0 \leq k \leq 4}. 
\end{align*}
Note that the convex hull of $ \Phi_{\mathrm{H}_{2}} $ is the regular decagon. 
One may choose simple roots as follows. 
\begin{align*}
\alpha_{1} = (1,0), \qquad
\alpha_{2} = \left(\cos\dfrac{4\pi}{5}, \sin\dfrac{4\pi}{5} \right). 
\end{align*}
The coefficient matrix of the positive roots with respect to $\{ \alpha_1, \alpha_2 \}$ is 
\begin{align*}
\begin{pmatrix}
1 & 0 & \tau & 1 & \tau \\
0 & 1 & 1 & \tau & \tau
\end{pmatrix}. 
\end{align*}
The LCM-period is $ \rho_{\mathrm{H}_{2}} = \langle 1 \rangle $ and the constituent of $ \chi_{\mathrm{H}_{2}}^{\mathrm{quasi}} $ is as follows: 
\begin{align*}
f_{\mathrm{H}_2}^{\langle 1 \rangle}(t) = t^{2}-5t+4 = (t-1)(t-4). 
\end{align*}

\subsection{The characteristic quasi-polynomial $\chi_{\mathrm{H}_{3}}^{\mathrm{quasi}}$}
Suppose that 
\begin{align*}
\Phi_{\mathrm{H}_{3}} = \Set{\begin{array}{cl}
(\pm 1, 0, 0) & \text{ and all permutations} \vspace{1mm} \\
\frac{1}{2}(\pm \tau, \pm 1, \pm \tau^{-1}) & \text{ and all even permutations}
\end{array} }. 
\end{align*}
Note that the convex hull of $ \Phi_{\mathrm{H}_{3}} $ is known as the icosidodecahedron. 
One may choose simple roots as follows: 
\begin{align*}
\alpha_{1} = \frac{1}{2}(\tau, -1, \tau^{-1}), \qquad
\alpha_{2} = \frac{1}{2}(-\tau, 1, \tau^{-1}), \qquad
\alpha_{3} = \frac{1}{2}(1, \tau^{-1}, -\tau). 
\end{align*}
The coefficient matrix of the positive roots with respect to $\{ \alpha_1, \alpha_2, \alpha_3 \}$ is 
\begin{align*}
\begin{pmatrix}
1 & 0 & \tau & 1 & \tau & 0 & 0 & \tau & \tau & \tau^{2} & 1 & \tau & \tau & \tau^{2} & \tau^{2} \\
0 & 1 & 1 & \tau & \tau & 0 & 1 & 1 & \tau^{2} & \tau^{2} & \tau & \tau & \tau^{2} & \tau^{2} & 2 \tau \\
0 & 0 & 0 & 0 & 0 & 1 & 1 & 1 & 1 & 1 & \tau & \tau & \tau & \tau & \tau
\end{pmatrix}. 
\end{align*}
The LCM-period is $ \rho_{\mathrm{H}_{3}} = \langle 2 \rangle $ and the constituents of $ \chi_{\mathrm{H}_{3}}^{\mathrm{quasi}} $ is as follows: 
\begin{align*}
f^{\langle 1 \rangle}_{\mathrm{H}_3} (t) &= t^{3}-15t^{2}+59t-45 = (t-1)(t-5)(t-9), \\
f^{\langle 2 \rangle}_{\mathrm{H}_3} (t) &= t^{3}-15t^{2}+59t-60 = (t-4)(t^{2}-11t+15). 
\end{align*}

\subsection{The characteristic quasi-polynomial $\chi_{\mathrm{H}_{4}}^{\mathrm{quasi}}$}
Suppose that 
\begin{align*}
\Phi_{\mathrm{H}_{4}} = \Set{\begin{array}{cl}
(\pm 1, 0, 0, 0) & \text{ and all permutations} \vspace{1mm} \\
\frac{1}{2}(\pm 1, \pm, 1, \pm 1, \pm 1) & \text{ and all permutations} \vspace{2mm} \\
\frac{1}{2}(\pm \tau, \pm 1, \pm \tau^{-1}, 0) & \text{ and all even permutations}
\end{array} }. 
\end{align*}
Note that the convex hull of $ \Phi_{\mathrm{H}_{4}} $ is known as the $ 600 $-cell. 
One may choose simple roots as follows: 
\begin{align*}
\alpha_{1} &= \frac{1}{2}(\tau, -1, \tau^{-1},0), \qquad
\alpha_{2} = \frac{1}{2}(-\tau, 1, \tau^{-1},0), \\
\alpha_{3} &= \frac{1}{2}(1, \tau^{-1}, -\tau,0), \qquad
\alpha_{4} = \frac{1}{2}(-1,-\tau,0,\tau^{-1}).  
\end{align*}
The coefficient matrix of the positive roots with respect to $\{ \alpha_1, \alpha_2, \alpha_3, \alpha_4 \}$ is 
\begin{align*}
&
\left(
\begin{array}{cccccccccc}
1 & 0 & 1 & \tau & \tau & 0 & 0 & \tau & \tau & \tau + 1 
\\
0 & 1 & \tau & 1 & \tau & 0 & 1 & 1 & \tau + 1 & \tau + 1 
\\
0 & 0 & 0 & 0 & 0 & 1 & 1 & 1 & 1 & 1 
\\
0 & 0 & 0 & 0 & 0 & 0 & 0 & 0 & 0 & 0 
\end{array}
\right.
\\
&
\left.
\ \ 
\begin{array}{cccccccccc}
1 & \tau & \tau & \tau + 1 & \tau + 1 & 0 & 0 & 0 & \tau & \tau 
\\
\tau & \tau & \tau + 1 & \tau + 1 & 2 \tau & 0 & 0 & 1 & 1 & \tau + 1 
\\
\tau & \tau & \tau & \tau & \tau & 0 & 1 & 1 & 1 & 1 
\\
0 & 0 & 0 & 0 & 0 & 1 & 1 & 1 & 1 & 1 
\end{array}
\right.
\\
&
\left. 
\ \ 
\begin{array}{cccccccccc}
\tau + 1 & \tau & \tau + 1 & \tau + 1 & 2 \tau + 1 & 2 \tau + 1 & 2 \tau + 1 & 1 & \tau & \tau 
\\
\tau + 1 & \tau + 1 & \tau + 1 & 2 \tau + 1 & 2 \tau + 1 & 2 \tau + 2 & 2 \tau + 2 & \tau & \tau & \tau + 1 
\\
1 & \tau + 1 & \tau + 1 & \tau + 1 & \tau + 1 & \tau + 1 & \tau + 2 & \tau & \tau & \tau 
\\
1 & 1 & 1 & 1 & 1 & 1 & 1 & \tau & \tau & \tau 
\end{array}
\right.
\\
&
\left.
\ \ 
\begin{array}{cccccccccc}
\tau + 1 & \tau + 1 & \tau & \tau + 1 & \tau + 1 & 2 \tau + 1 & 2 \tau + 1 & \tau + 1 & \tau + 1 & 2 \tau + 1 
\\
\tau + 1 & 2 \tau & \tau + 1 & \tau + 1 & 2 \tau + 1 & 2 \tau + 1 & 2 \tau + 2 & 2 \tau & 2 \tau + 1 & 2 \tau + 1 
\\
\tau & \tau & \tau + 1 & \tau + 1 & \tau + 1 & \tau + 1 & \tau + 1 & 2 \tau & 2 \tau & 2 \tau 
\\
\tau & \tau & \tau & \tau & \tau & \tau & \tau & \tau & \tau & \tau 
\end{array}
\right.
\\
&
\left.
\ \ 
\begin{array}{cccccccccc}
2 \tau + 1 & 2 \tau + 2 & 2 \tau + 1 & 2 \tau + 1 & 2 \tau + 2 & 3 \tau + 1 & 2 \tau + 2 & 2 \tau + 1 & 2 \tau + 1 & 2 \tau + 1 
\\
3 \tau + 1 & 3 \tau + 1 & 2 \tau + 2 & 3 \tau + 1 & 3 \tau + 1 & 3 \tau + 2 & 3 \tau + 2 & 2 \tau + 2 & 2 \tau + 2 & 3 \tau + 1 
\\
2 \tau & 2 \tau & 2 \tau + 1 & 2 \tau + 1 & 2 \tau + 1 & 2 \tau + 1 & 2 \tau + 1 & \tau + 2 & 2 \tau + 1 & 2 \tau + 1 
\\
\tau & \tau & \tau & \tau & \tau & \tau & \tau & \tau + 1 & \tau + 1 & \tau + 1 
\end{array}
\right.
\\
&
\left.
\ \ 
\begin{array}{cccccccccc}
2 \tau + 2 & 2 \tau + 2 & 3 \tau + 1 & 2 \tau + 2 & 3 \tau + 1 & 3 \tau + 1 & 3 \tau + 2 & 3 \tau + 2 & 3 \tau + 2 & 3 \tau + 2 
\\
3 \tau + 1 & 3 \tau + 2 & 3 \tau + 2 & 3 \tau + 2 & 3 \tau + 2 & 3 \tau + 3 & 3 \tau + 3 & 4 \tau + 2 & 4 \tau + 2 & 4 \tau + 2 
\\
2 \tau + 1 & 2 \tau + 1 & 2 \tau + 1 & 2 \tau + 2 & 2 \tau + 2 & 2 \tau + 2 & 2 \tau + 2 & 2 \tau + 2 & 3 \tau + 1 & 3 \tau + 1 
\\
\tau + 1 & \tau + 1 & \tau + 1 & \tau + 1 & \tau + 1 & \tau + 1 & \tau + 1 & \tau + 1 & \tau + 1 & 2 \tau
\end{array}
\right). 
\end{align*}

The LCM-period is $\rho_{\mathrm{H}_{4}} = \langle 6\sqrt{5} \rangle $ and the constituents of $ \chi_{\mathrm{H}_{4}}^{\mathrm{quasi}} $ is as follows: 
\begin{align*}
f_{\mathrm{H}_{4}}^{\langle 1 \rangle}(t) &= t^{4}-60t^{3}+1138t^{2}-7140t+6061 = (t-1)(t-11)(t-19)(t-29), \\
f_{\mathrm{H}_{4}}^{\langle 3\rangle}(t) &= t^{4}-60t^{3}+1138t^{2}-7140t+9261 = (t-9)(t-21)(t^{2}-30t+49),   \\
f_{\mathrm{H}_{4}}^{\langle \sqrt{5}\rangle}(t) &= t^{4}-60t^{3}+1138t^{2}-7140t+14125 = (t-5)(t-25)(t^{2}-30t+113),  \\
f_{\mathrm{H}_{4}}^{\langle 3\sqrt{5}\rangle}(t) &= t^{4}-60t^{3}+1138t^{2}-7140t+17325,  \\
f_{\mathrm{H}_{4}}^{\langle 2\rangle}(t) &= t^{4}-60t^{3}+1138t^{2}-8040t+17536 = (t-4)(t-16)(t^{2}-40t+274) ,\\
f_{\mathrm{H}_{4}}^{\langle 6\rangle}(t) &= t^{4}-60t^{3}+1138t^{2}-8040t+20736, \\
f_{\mathrm{H}_{4}}^{\langle 2\sqrt{5}\rangle}(t) &= t^{4}-60t^{3}+1138t^{2}-8040t+25600 = (t-20)(t^{3}-40t^{2}+338t-1280), \\
f_{\mathrm{H}_{4}}^{\langle 6\sqrt{5}\rangle}(t) &= t^{4}-60t^{3}+1138t^{2}-8040t+28800.
\end{align*}

\subsection{Observations}

Kamiya, Takemura, and Terao \cite[Theorem 3.1]{kamiya2010characteristic-alsas} gave an explicit formula of the generating function $ \Gamma_{\Phi} \coloneqq \sum_{q = 1}^{\infty} \chi_{\Phi}^{\mathrm{quasi}}(q)t^{q} $ for an irreducible crystallographic root system $ \Phi $ in terms of the coefficient of the highest root and the Coxeter number. 
We obtain the following corollary. 

\begin{corollary}[Kamiya-Takemura-Terao {\cite[Corollary 3.4]{kamiya2010characteristic-alsas}}] \label{positive <=> q geq h}
Let $ h $ be the Coxeter number of $ \Phi $. 
Then $ \chi_{\Phi}^{\mathrm{quasi}}(q) > 0 $ if and only if $ q \geq h $. 
\end{corollary}

In order to consider the analogue, we list the Coxeter numbers and possible values of the absolute norms of nonzero ideals of $ \mathcal{O} $. 
The Coxeter numbers are 
\begin{align*}
h_{\mathrm{H}_{2}} = 5, \qquad
h_{\mathrm{H}_{3}} = 10, \qquad 
h_{\mathrm{H}_{4}} = 30. 
\end{align*}

The possible values of absolute norms are positive integers of the form $ a^{2}+ab-b^{2} $: 
\begin{align*}
1,4,5,9,11,16,19,20,25,29,31, \dots. 
\end{align*}

From the lists above, we obtain the following theorem. 
\begin{theorem}
Let $ \ell \in \{2,3,4\} $ and $ \mathfrak{a} $ a nonzero ideal of $ \mathcal{O} $. 
Then $ \chi_{\mathrm{H}_{\ell}}^{\mathrm{quasi}}(\mathfrak{a}) > 0 $ if and only if $ N(\mathfrak{a}) \geq h_{\mathrm{H}_{\ell}} $. 
\end{theorem}

It is well known that the exponents of a root system satisfy duality with respect to the Coxeter number. 
Namely, if $d$ is an exponent, then $h-d$ is also an exponent, where $h$ denotes the Coxeter number. 
The exponents for $\mathrm{H}_{2}, \mathrm{H}_{3}$, and $\mathrm{H}_{4}$ are 
\begin{align*}
(1,4), \qquad
(1,5,9), \qquad 
(1, 11, 19, 29), 
\end{align*}
respectively. 
They appear in the factorization of the characteristic polynomials. 

The characteristic quasi-polynomial of an irreducible crystallographic root system also has duality with respect to the Coxeter number. 
The duality can be shown from the explicit expressions given by Kamiya, Takemura, and Terao \cite{kamiya2007characteristic-a}, or Suter \cite{suter1998number-moc}. 
Yoshinaga \cite{yoshinaga2018worpitzky-tmj} gave a classification-free proof. 

\begin{theorem}[Yoshinaga {\cite[Corollary 3.8]{yoshinaga2018worpitzky-tmj}}] \label{duality}
Let $ \Phi $ be an irreducible crystallographic root system of rank $ \ell $ and $ h $ its Coxeter number. 
Then $ \chi_{\Phi}^{\mathrm{quasi}}(q) = (-1)^{\ell}\chi_{\Phi}^{\mathrm{quasi}}(h-q) $. 
\end{theorem}

Note that the duality holds as quasi-polynomials but not the level of the constituents. 
Yoshinaga \cite{yoshinaga2018characteristic-joctsa} studied the condition for the constituents to hold the duality in detail. 

We consider an analogue of Theorem \ref{duality} for our cases. 
One can see that if $ \kappa \mid \rho_{\mathrm{H}_{\ell}} $ is not a multiple of $ \langle 2 \rangle $, then the $ \kappa $-constituent satisfies the duality. 
Namely $ f_{\mathrm{H}_{\ell}}^{\kappa}(t) = (-1)^{\ell}f_{\mathrm{H}_{\ell}}^{\kappa}(h_{\mathrm{H}_{\ell}}-t) $ when $ \langle 2 \rangle \nmid \kappa $. 
However, unfortunately it seems that the $ \kappa $-constituent does not satisfy the duality when $ \kappa $ is a multiple of $ \langle 2 \rangle $. 

In the proof of Theorem \ref{duality}, the coweight lattice plays an important role. 
The coweight lattice is generated freely by the dual basis of the the simple roots. 
Hence it may make sense to consider the coefficients of the positive roots with respect to the dual basis of the simple roots in our cases. 

The change-of-basis matrices for $ \mathrm{H}_{3} $ and $ \mathrm{H}_{4} $ are 
\begin{align*}
\begin{pmatrix}
1 & -\frac{\tau}{2} & 0 \\
-\frac{\tau}{2} & 1 & -\frac{1}{2} \\
0 & -\frac{1}{2} & 1
\end{pmatrix} \qquad \text{ and } \qquad
\begin{pmatrix}
1 & -\frac{\tau}{2} & 0 & 0 \\
-\frac{\tau}{2} & 1 & -\frac{1}{2} & 0 \\
0 & -\frac{1}{2} & 1 & -\frac{1}{2} \\
0 & 0 & -\frac{1}{2} & 1
\end{pmatrix}. 
\end{align*}

Therefore in this case the natural coefficient ring may be $ \mathbb{Z}[\tau, \frac{1}{2}] $, which is a localization of $ \mathcal{O} = \mathbb{Z}[\tau] $ with respect to $ S = \Set{2^{m} \in \mathcal{O} | m \in \mathbb{Z}_{\geq 0}} $. 
By Theorem \ref{localization}, the $ \kappa $-constituents with $ \langle 2 \rangle \mid \kappa $ are eliminated and only the constituents that satisfy the duality remain. 

\begin{question}
Are there any reasons for these results? 
\end{question}

\section*{Acknowledgments}

The authors thank Professor Masahiko Yoshinaga for his comments about the duality of the characteristic quasi-polynomials. 
This work was supported by JSPS KAKENHI Grant Number JP22K13885.

\bibliographystyle{amsplain}
\bibliography{bibfile}

\end{document}